\documentclass{svjour3}                   
\smartqed  
\usepackage{graphicx}
\usepackage{subfigure}
\usepackage{epstopdf} 
\usepackage{amsmath}
\usepackage{amsfonts}
\usepackage{mathrsfs}
\usepackage{amsfonts}

\usepackage{cases}
\usepackage{booktabs}
\usepackage{multirow}

\usepackage{algorithmic}
\usepackage{cite}
\usepackage{adjustbox}
\usepackage{blindtext}
\usepackage{CJK}
\usepackage{bm}
\usepackage{lipsum}

\usepackage{mathrsfs,enumerate,color}
\usepackage{appendix}
\usepackage{float}
\usepackage{cite}

\usepackage{amsopn}

\newtheorem{algorithm}{Algorithm}
\numberwithin{equation}{section}
\numberwithin{table}{section}
\journalname{AAA}
\begin{document}

\title{Randomized Quaternion Singular Value Decomposition  for  Low-Rank Approximation}

\subtitle{\it Dedicated to Professor Musheng Wei on the occasion of his 75th birthday}


\author{Qiaohua Liu  \and Sitao Ling \and Zhigang Jia}


\institute{Qiaohua Liu \at
              Department of Mathematics, Shanghai University, Shanghai 200444, P.R. China \\
              Supported by the National Natural Science Foundation of China under grant 11001167.\\
              \email{qhliu@shu.edu.cn}           
              \and
              Sitao Ling  \at
              School of Mathematics, China University of Mining and Technology, Xuzhou, Jiangsu,
221116, P.R. China.\\
         \email{lingsitao2004@163.com}
          \and
           Zhigang Jia (Corresponding author) \at
           School of Mathematics and Statistics, Jiangsu Normal University,
Xuzhou 221116, P. R. China\\
Research Institute of Mathematical Science, Jiangsu Normal University,
Xuzhou 221116, P. R. China\\
Supported by National Natural Science Foundation
of China  grants 12171210, 12090011 and 11771188; the Major Projects of Universities in Jiangsu Province (No. 21KJA110001); the Priority Academic Program Development Project (PAPD);  the Top-notch Academic Programs Project (No. PPZY2015A013) of Jiangsu Higher Education Institutions.\\
              \email{zhgjia@jsnu.edu.cn}
}

\date{Received: date / Accepted: date}

\maketitle

\begin{abstract}
This paper presents a randomized  quaternion singular value decomposition (QSVD) algorithm for  low-rank matrix approximation problems, which are widely used in color face recognition, video compression, and signal processing problems.  With   quaternion normal distribution based  random sampling, the randomized QSVD algorithm projects a high-dimensional data to a low-dimensional subspace and then identifies an approximate range subspace of the quaternion matrix. The key statistical properties of quaternion Wishart distribution are proposed and  used to perform the approximation error analysis of the algorithm.  Theoretical results show that the randomized QSVD algorithm can trace  dominant  singular value decomposition triplets of a quaternion matrix with acceptable accuracy. Numerical experiments also indicate  the rationality  of proposed theories.
Applied to color face recognition problems, the randomized QSVD algorithm obtains  higher recognition accuracies and behaves   more efficient than the known Lanczos-based partial QSVD and a quaternion version of fast frequent directions algorithm.

\keywords{randomized quaternion SVD;  quaternion Wishart distribution;  low-rank approximation; error analysis.}
\end{abstract}


\section{Introduction} \label{intro-sec}
Low-rank approximations of quaternion matrices  play an important role in color image processing area \cite{jns19nla,jnw19}, in which color images are represented by pure quaternion matrices. Based on the color principal component analysis \cite{zjccg21},  the optimal rank-$k$ approximations preserve the main features and the important low frequency information of original color image samples. The core work of generating low-rank approximations is to compute  the dominant quaternion singular value decomposition (QSVD) triplets (i.e., left singular vectors, singular values and right singular vectors).  However, there are still few efficient algorithms to do this work when quaternion matrices are of large-scale sizes. No  rigorous error analysis of computed approximations have also been given in the literature. In this paper, we present a new randomized QSVD algorithm and propose important theoretical results about the feasibility and the reliability of the algorithm.

In these years, quaternions \cite{ha} and  quaternion matrices \cite{zh} have been more and more attractive in  many research fields such as signal processing \cite{ebs14},   image data analysis \cite{bs,jing21}, and machine learning \cite{minm17,zjccg21}.  Because of non-commutative multiplication  of quaternions,  quaternion matrix computations contain more abundant challenging topics than  real or complex matrix computations. The algorithms designed for quaternion matrices are also feasible for the real or complex case, but the converse is not always true. As we are concerned on, QSVD triplets can be achieved in three totally different ways. The first one is to call the  {\sf svd} command from  Quaternion toolbox  for Matlab (QTFM) developed by Sangwine and  Bihan in 2005.  For the principle of the algorithm, we refer to  \cite{bs2}. The codes in QTFM  are based on quaternion arithmetic operations and is less efficient for large matrices.
The second one is to use the real structure-preserving QSVD method \cite{wl}. Its main idea is to perform real operations on  the real  counterparts of quaternion matrices with structure preserving scheme. In practical implementations, only  the first block row or column of the real counterpart is explicitly stored and updated,  and  the other subblocks are implicitly formulated  with the aid of the algebraic symmetry structure.  The real matrix-matrix multiplication-based  BLAS-3 operations  make the computation more efficient. The concept of structure-preserving  was firstly proposed to solve quaternion eigenvalue problem  in \cite{jwl}, and then extended to the computations of quaternion LU \cite{lwz3,wm1} 
 and QR  \cite{lwz2}  factorizations. Recently, Jia et al. \cite{jwzc} developed a new structure-preserving  quaternion QR algorithm for eigenvalue problems of general quaternion matrices, by constructing feasible frameworks of calculation for new quaternion Householder reflections and generalized  Givens transformations.
 For more issues about structure-preserving algorithms, we refer to  two monographs  \cite{wl} by Wei et al. and \cite{jia2019} by Jia.
The above two ways are based on the truncation of the full QSVD and the computational cost is expensive in computing all singular values and corresponding left and right singular vectors.
Thus they are not feasible for large-scale quaternion matrices.  Jia et al. \cite{jns} proposed  a promising  iterative algorithm to compute  dominant QSVD triplets, based on the Lanczos bidiagonalization \cite{gv2}  with reorthogonalization and thick-restart techniques. This method is referred to as the {\sf lansvdQ} method.
The superiority of {\sf lansvdQ} method over the full QSVD was revealed in \cite{jns}, through a number of practical applications such as color face recognition, video compression and color image completion. When the target rank $k$ increases, the matrix-vector products at each iteration of {\sf lansvdQ} make the computational cost   increase.  Is there any method with lower computational cost for the quaternion low-rank approximation problem?

In the past decade, randomized algorithms for computing approximations of real matrices  have been receiving more and more attention.  Randomized projection and randomized sampling  are two commonly used techniques to  deal with large-scale problems efficiently.
Randomized projection combines rows or columns together to produce a small sketch of $M\in {\mathbb R}^{m\times n} (m\ge n)$ \cite{sa2}.  Possible techniques include  subspace iterations \cite{gu}, subspace embedding ({\sf SpEmb})  \cite{mm}, frequent directions ({\sf FD})\cite{glp} and etc. Recently, Teng and Chu  \cite{tc} implanted {\sf SpEmb} in {\sf FD} to develop  a  fast frequent direction  ({\sf SpFD}) algorithm. Through the experimental results on  world datasets  and applications in network analysis, the superiority of {\sf SpFD} over {\sf FD} is displayed, not only in the efficiency, but also in the effectiveness.

Randomized sampling finds a small subset of rows or columns based on a pre-assigned probability distribution, say, by pre-multiplying  $M$ on
 an $n\times \ell ~(\ell \ll n)$ random Gaussian matrix $\Omega$, and identifies a low-dimensional approximate range subspace of $M$, after which a small-size matrix approximation is also obtained. The idea of a randomized sampling procedure can be traced to a  2006 technical report of paper  \cite{mrt},  and later analyzed and elaborated in \cite{dg,gu,hmt,ma,sa,tc,ygl,zw}.
They are computationally efficient for large-scale problems and    adapt to the case that the numerical rank is known or can be estimated in advance. When the singular values have relatively fast decay rate, the algorithm is inherently stable. For singular values with slow decay rate, the randomized algorithm with power scheme will enhance the stability of the algorithm.

In this paper we consider the randomized sampling algorithm for quaternion low-rank matrix approximations.
 The targeted randomized QSVD algorithm is expected to have lower computational cost and to be appropriate for choosing a small number of dominant QSVD triplets of large-scale quaternion matrices.
It seems natural to utilize the research framework in   \cite{hmt} and generalize the real randomized SVD algorithm  to quaternion matrices.   Unfortunately,  the theoretical analysis is long and arduous. It involves doses of statistics related to quaternion variables and several difficulties block us to go further.
\begin{itemize}
\item What kind of quaternion distribution is appropriate for the randomized QSVD algorithm?
The proper quaternion  distribution should be
invariant under unitary transformations, which will bring convenience for approximation error analysis
of the proposed algorithm.
  However, few studies have been seen on the probability distribution
of quaternion variables in the literature.

\item What are the  distributions   of  the norms of  the pseudoinverse ${\bf \Omega}^\dag$ of quaternion random Gaussian matrix ${\bf \Omega}$?
Due to  the   non-commutative multiplication of quaternions,   quaternion determinant and integrals   could not be defined similar to the real case.
 Hence,  real probability theories   could not be directly used  to evaluate the norms of quaternion random Gaussian  matrices.

\item What are statistical evaluations of    spectral norms of ${\bf \Omega}$  and its real counterpart?
  The real counter part   $\Upsilon_{{\bf \Omega}}$ (see (\ref{2.1})) is  a non-Gaussian random matrix.  
 It is necessary to develop  novel techniques to evaluate the expectation and probability bounds of   $\|{\bf \Omega}\|_2$ and its scaled norms.
 \end{itemize}

Based on the investigations on  key features of  ${\bf \Omega}$,  we will give expectation and deviation bounds  for approximation errors of the quaternion randomized SVD algorithm.
To the best of our knowledge, these results are new and  no developments have been made on the proposed algorithm and theories about quaternion matrix approximation problems.
With high probability, the theoretical results show that the low rank approximations can be computed quickly for quaternion matrices with rapidly decaying singular values.  Through the  numerical experiments, the superiority of  the proposed algorithm will be displayed, in comparison with the quaternion Lanczos  method and a quaternion version of  {\sf SpFD} \cite{tc}.

The paper is organized as follows. In Section 2, we review some preliminary results about quaternion matrices and randomized SVD for real matrices.
The  randomized QSVD algorithm and implement details for  low-rank approximation problems will be  studied in Section 3.
In Section 4,  the theoretical analysis is provided for the approximation errors.
In Section 5,  we test the theories and  numerical behaviors of the proposed algorithms through several experiments and show their efficiency over  Lanczos-based partial QSVD algorithm
and quaternion {\sf SpFD} for  color face recognition problems.

Throughout this paper, we denote by ${\mathbb R}^{m\times n}$ and  ${\mathbb Q}^{m\times n}$ the spaces  of all $m\times n$ real and quaternion matrices, respectively. The norm $\|\cdot\|_a$ denotes either the spectral norm or the Frobenius norm.
For quaternion matrix ${\bf A}\in {\mathbb Q}^{m\times n}$,
  ${\bf A}^\dag$ is  the pseudoinverse of ${\bf A}$, and ${\cal R}({\bf A})$ represents the column range space of ${\bf A}$.   {\rm tr}($\cdot$) denotes the trace of a quaternion or real square matrix, and etr($\cdot$)=exp(tr($\cdot$)) means the exponential operation of the trace.
Let ${\sf P}\{\cdot\}$  denote the probability of an event and   ${\sf E}(\cdot)$ denote  the expectation of a random variable.
 For  differentials ${\rm d}y_1, {\rm d}y_2$ of  real random variables $y_1, y_2$,   ${\rm d}y_1\wedge {\rm d}y_2$ denotes  the  non-commutative  exterior product of ${\rm d}y_1, {\rm d}y_2$, under which ${\rm d}y_1\wedge {\rm d}y_2=-{\rm d}y_2\wedge {\rm d}y_1$ and ${\rm d}y_1\wedge {\rm d}y_1=0$.

\section{Preliminaries}

In this section, we first introduce some   basic information of quaternion matrices and   quaternion SVD. The basic randomized SVD  for   real matrices is described thereafter.

 \subsection{Quaternion matrix and QSVD}
The quaternion skew-field ${\mathbb Q}$ is an associative but non-commutative algebra of rank four over ${\mathbb R}$, and any quaternion ${\bf q}\in {\mathbb Q}$  has one real part and three imaginary parts given by
$
{\bf q}=q_0+q_1{\bf i}+q_2{\bf j}+q_3{\bf k},
$
where $q_0, q_1, q_2, q_3\in {\mathbb R}$, and {\bf i}, {\bf j} and {\bf k} are three imaginary units satisfying
$
{\bf i}^2={\bf j}^2={\bf k}^2={\bf i}{\bf j}{\bf k}=-1.
$
The conjugate and modulus of ${\bf q}$ are  defined by
$
{\bf q}^*=q_0-q_1{\bf i}-q_2{\bf j}-q_3{\bf k}$
and
$|{\bf q}|=\sqrt{q_0^2+q_1^2+q_2^2+q_3^2}$, respectively.

For any quaternion matrices ${\bf P}=P_0+P_1{\bf i}+P_2{\bf j}+P_3{\bf k}\in {\mathbb Q}^{m\times n}$, ${\bf Q}=Q_0+Q_1{\bf i}+Q_2{\bf j}+Q_3{\bf k}\in {\mathbb Q}^{m\times n}$,
denote ${\bf Q}^*=Q_0^T-Q_1^T{\bf i}-Q_2^T{\bf j}-Q_3^T{\bf k}$ and the sum of  ${\bf P}, {\bf Q}$ as
 $
 {\bf P}+{\bf Q}=(P_0+Q_0)+(P_1+Q_1){\bf i}+(P_2+Q_2){\bf j}+(P_3+Q_3){\bf k},
 $
and  for  quaternion matrix ${\bf S}\in {\mathbb Q}^{n\times \ell}$, the multiplication ${\bf QS}$ is given by
 \begin{eqnarray*}
&(Q_0S_0-Q_1S_1-Q_2S_2-Q_3S_3)+(Q_0S_1+Q_1S_0+Q_2S_3-Q_3S_2){\bf i}+\\
&(Q_0S_2-Q_1S_3+Q_2S_0+Q_3S_1){\bf j}+(Q_0S_3+Q_1S_2-Q_2S_1+Q_3S_0){\bf k}.
\end{eqnarray*}

For ${\bf Q}\in {\mathbb Q}^{m\times n}$, define the real counterpart $\Upsilon_{\bf Q}$  and the column representation ${\bf Q}_{\rm c}$ as
\begin{equation}
\Upsilon_{\bf Q}=\left[\begin{array}{rrrr}
Q_0&-Q_1&-Q_2&-Q_3\\ Q_1&Q_0&-Q_3&Q_2\\
Q_2&Q_3&Q_0&-Q_1\\ Q_3&-Q_2&Q_1&Q_0
\end{array}\right],\qquad {\bf Q}_{\rm c}=\left[\begin{array}c Q_0\\Q_1\\Q_2\\Q_3\end{array}\right].
\label{2.1}
\end{equation}
Note that $\Upsilon_{\bf Q}$ has special real algebraic structure that is preserved under the following operations\cite{jwl,lwz2}:
\begin{equation}
\Upsilon_{k_1{\bf P}+k_2{\bf Q}}=k_1\Upsilon_{{\bf P}}+k_2\Upsilon_{{\bf Q}}  ~ (k_1, k_2\in {\mathbb R}),~~
 \Upsilon_{{\bf Q}^*}=\Upsilon_{{\bf Q}}^T, \qquad
 \Upsilon_{{\bf QS}}=\Upsilon_{{\bf Q}}\Upsilon_{{\bf S}}.
 \label{2.3}
 \end{equation}

For determinants of   quaternion square  matrices,  a variety of definitions have emerged  in terms of the complex and real  counterparts to avoid the difficulties caused by the non-commutativity of quaternion multiplications;  see \cite{lsm,ro,zh} and reference therein.
 However these  definitions do not coincide with the standard determinant of a real matrix.
In this paper, we  only consider the determinant of  Hermitian quaternion matrices, which was defined by Li \cite{lsm} as
 \begin{equation}
{\bf det}({\bf Q})=\lambda_1\lambda_2\cdots\lambda_n, \quad {\bf Q}\in {\mathbb Q}^{n\times n} \mbox{ is Hermitian, }\label{2.30}
\end{equation}
where $\lambda_1,\ldots, \lambda_n$ are eigenvalues of ${\bf Q}$, and they are proved to be real \cite{jwl,lsm}.  This definition   in (\ref{2.30}) is consistent with the determinant of a real symmetric matrix, but   does not adapt to the quaternion non-Hermitian matrices, since
   a quaternion non-Hermitian matrix has significantly different properties in its left and right eigenvalues, and
   there is no very close relation between left and right eigenvalues \cite{zh}. When ${\bf Q}$ is Hermitian, the left and right eigenvalues are coincided to be the same  real value.  Throughout this paper we use ${\bf det}({\bf Q})$   to distinguish it from the real determinant symbol ``det''.
 Moreover, if ${\bf Q}$ is positive semidefinite so that $\lambda_i\ge 0$, then the quaternion determinant ${\bf det}({\bf Q})$
  can be represented in terms of a determinant of a real matrix \cite{lsm} as
\begin{equation}
 {\bf det}({\bf Q})=\left[\det(\Upsilon_{\bf Q})\right]^{1/4},\quad {\bf Q} \mbox{ is Hermitian and positive semidefinite. }\label{2.4}
\end{equation}

\begin{definition}\label{d:2-2}
 The spectral norm (2-norm) of a quaternion vector ${\bf x}=[{\bf x}_i]\in {\mathbb Q}^n$ is $\|{\bf x}\|_2:=\sqrt{\sum_i|{\bf x}_i|^2}$. The  2-norm  of a quaternion matrix ${\bf A}=[{\bf a}_{ij}]\in {\mathbb Q}^{m\times n}$ are
  $\|{\bf A}\|_2:=\max\sigma({\bf A})$,    where $\sigma({\bf A})$ is the set of singular values of ${\bf A}$. The Frobenius norm of ${\bf A}$ is
   $\|{\bf A}\|_F=\big(\sum\limits_{i,j}|{\bf a}_{ij}|^2\big)^{1/2}=\left[{\rm tr}({\bf A}^*{\bf A})\right]^{1/2}.$
\end{definition}

QSVD was firstly proposed in \cite[Theorem 7.2]{zh}  and the partial QSVD was presented in \cite{jns}.
\begin{lemma}[\hspace{-0.02em}QSVD \cite{zh}]\label{l:2-1}
 Let ${\bf A}\in {\mathbb Q}^{m\times n}$.  Then there exist two quaternion unitary matrices ${\bf U}\in {\mathbb Q}^{m\times m}$ and ${\bf V}\in {\mathbb Q}^{n\times n}$ such that ${\bf U}^*{\bf AV}=\Sigma$, where $\Sigma={\rm diag}(\sigma_1,\sigma_2,\ldots,\sigma_{l})\in {\mathbb R}^{m\times n}$  with $\sigma_i\ge 0$ denoting the $i$-th largest singular value of ${\bf A}$ and $l=\min (m,n)$.
\end{lemma}
From \cite{jns},  the optimal rank-$k$ approximation of ${\bf A}$ is given by
 $
 {\bf A}_k={\bf U}_k{\Sigma}_k{\bf V}_k^*,
 $
 where ${\bf U}_k$ and ${\bf V}_k$ are respectively submatrices of ${\bf U}$ and ${\bf V}$  by taking their first $k$ columns, and ${\Sigma }_k={\rm diag}(\sigma_1,\ldots, \sigma_k)$.
Furthermore, by the real counterpart of QSVD: $\Upsilon_{\bf U}^T\Upsilon_{\bf A}\Upsilon_{\bf V}=\Upsilon_\Sigma$, where $\Upsilon_{\bf U}$ and $\Upsilon_{\bf V}$ are real orthogonal matrices, and
 $\Upsilon_\Sigma={\rm diag}(\Sigma, \Sigma, \Sigma, \Sigma)$.  As a result,  spectral and Frobenius norms of a quaternion matrix can be represented by the ones of real matrices as below
\begin{equation}
\|{\bf A}\|_2=\|\Upsilon_{\bf A}\|_{2},\qquad
\|{\bf A}\|_F={1\over 2}\|\Upsilon_{\bf A}\|_{F}=\|{\bf A}_{\rm c}\|_F.\label{2.5}
\end{equation}
Moreover, for consistent quaternion matrices ${\bf A}$ and ${\bf B}$, it is obvious that
\begin{equation}
\|{\bf AB}\|_F\le \|{\bf A}\|_2\|{\bf B}\|_F,\quad \|{\bf AB}\|_F\le \|{\bf A}\|_F\|{\bf B}\|_2.\label{2.6}
\end{equation}

\subsection{Real randomized SVD and low-rank approximation}


Given a real matrix  $M\in\mathbb{R}^{m\times n}$, randomized sampling methods  \cite{hmt,lwm,ma,mrt,wlr} apply the input matrix $M$ onto a diverse set of random sample vectors $\Omega=[\omega_1~\ldots~ \omega_{\ell}]$, expecting  $M\Omega$ to  capture the main information of the range space of $M$ and to maintain safe approximation error bounds with high probability.
In \cite{hmt}, a random Gaussian matrix $\Omega$ is used.   By applying $M$ to $\Omega$, and then computing the orthogonormal basis $Q$ of the range space of $M\Omega$ via skinny QR factorization in Matlab:
$$\Omega=\texttt{randn}(n,\ell), \qquad [Q,\sim]=\texttt{ qr}(Y,0),\qquad \text{where}~ Y=M\Omega,$$
 one can get an approximate orthogonal range space of $M$. Here $\ell=k+p$ and $p$ is a small oversampling factor (say, $p=5$).
In this case, the matrix $M$  is approximated by $M\approx QN,$
where  $QQ^T$ is an orthogonal projector and the matrix $N:=Q^TM$ is of small size $\ell\times n$.   The problem then reduces to
compute the full SVD of $N$ as $N=\hat U\hat S\hat V^T$. Therefore
$M\approx QN=Q\hat U\hat S\hat V^T,$
and once a suitable rank $k$ has been chosen based on the decay of $\hat S$,   the low-rank SVD factors can be determined as
$$
\bar U_k=Q\hat U(:,1:k),\qquad \bar S_k=\hat S(1:k,1:k),\quad \mbox{and}\quad \bar V_k=\hat V(:,1:k)
$$
such that $M_k\approx \bar U_k\bar S_k\bar V_k^T$. We refer to the above method as the randomized SVD.

The idea  is simple, but whether the projection $QQ^T$ can capture the range of $M$ well    depends not only on the property of random matrix, but also  on the singular values $s_i$ of the matrix $M$ we are dealing
with.
It was shown in  \cite[Theorems 10.5 and 10.6]{hmt}  that for $p\ge 2$,
the expectation of the approximation error satisfies
\begin{equation}
\begin{array}l
{\sf E}\|(I-QQ^T)M\|_{2}\le \Big(1+\sqrt{k\over p-1}\Big)s_{k+1}+{{\rm e}\sqrt{k+p}\over p}\Big(\sum\limits_{j=k+1}^{\min(m,n)}s_j^2\Big)^{1/2},\\
{\sf E}\|(I-QQ^T)M\|_{F}\le \Big(1+{k\over p-1}\Big)^{1/2}\Big(\sum\limits_{j=k+1}^{\min(m,n)}s_j^2\Big)^{1/2}.
\end{array}\label{2.9}
\end{equation}

It is observed that when the singular values of $M$ decay very slowly, the method fails to work
  well, because the singular vectors associated with the tail
singular values 
 capture a significant fraction
of the range of $M$, and   the range of  $Y=M\Omega$ as well.
Power scheme can be used to enhance the effect of the approximation, i.e., by applying power operation to generate $Y=(MM^T)^qM\Omega$,  where $(MM^T)^qM$ has the same singular space as  $M$, but  with a faster decay rate in its singular values.

{}

\section{Quaternion randomized SVD}
\setcounter{equation}{0}

In this section, we   develop the randomized QSVD ({\sf randsvdQ}) algorithm in Algorithm \ref{alg3.1} and  present   some measures to  improve the efficiency of the algorithm in practical implementations.

How to choose the random test matrix in the  algorithm? Consider a simple case about the
rank-1 approximation  ${\bf A}_1=\sigma_1{\bf u}_1{\bf v}_1^*$ of the quaternion matrix ${\bf A}$.
It is easy to prove that $\{{\bf y}_*, {\bf z}_*\}=\{{\bf u}_1, {\bf v}_1\}$ is the maximizer of  $\max\limits_{\|{\bf y}\|_2=\|{\bf z}\|_2=1} |{\bf y}^*{\bf A z}|$, and
$|{\bf y}^*\tilde{\bf z}|=|{\bf y}^*{\bf A z}|$ approximates $\sigma_1$ for ${\bf y}={\bf u}_1$ and $\tilde {\bf z}={\bf Av}_1\in {\cal R}({\bf A})$, in which   the columns of ${\bf A}$ are spanned with quaternion coefficients. 
In order to capture the main information of ${\cal R}({\bf A})$ spanned by dominant left singular vectors of ${\bf A}$,
it is natural to use  a set of $n\times 1$  quaternion random  vectors  ${\bf \Omega}=[\boldsymbol{\omega}^{(1)}~ \ldots~ \boldsymbol{\omega}^{(\ell)}]$  to span the columns of ${\bf A}$, with random standard  real Gaussian matrices  as the four parts of ${\bf \Omega}$. That means the $n\times \ell$ quaternion random test matrix
\begin{equation}
{\bf \Omega}={  \Omega}_0+{  \Omega}_1{\bf i}+{  \Omega}_2{\bf j}+{  \Omega}_3{\bf k},\label{3.1}
\end{equation}
 where the entries of ${ \Omega}_0, { \Omega}_1,{ \Omega}_2, {  \Omega}_{3}$ are random and independently  drawn from the $N(0,1)$-normal distribution. The detailed description of randomized QSVD is given in Algorithm \ref{alg3.1}.

\begin{algorithm}[({\sf randsvdQ}) Randomized QSVD with fixed rank]
\label{alg3.1}

(1) Given ${\bf A}\in {\mathbb Q}^{m\times n}$,  choose target rank $k$,  oversampling parameter $p$ and the power scheme parameter  $q$.
Set $\ell=k+p$, and draw an  $n\times \ell$  quaternion random test matrix ${\bf \Omega}$ as in (\ref{3.1}).

(2) Construct ${\bf Y}_0={\bf A}{\bf \Omega}$ and for $i=1,2,\ldots, q$, compute

$$\hat {\bf Y}_i={\bf A}^*{\bf Y}_{i-1} \quad\mbox{ and }\quad  {\bf Y}_i={\bf A}\hat{\bf Y}_{i}.$$

(3) Construct  an  ${m\times \ell}$ quaternion orthonormal basis ${\bf Q}$ for the range of ${\bf Y}_q$   by the quaternion QR  decomposition
and generate ${\bf B}={\bf Q}^*{\bf A}$.

(4) Compute the QSVD of a small-size matrix ${\bf B}$: ${\bf B}={  {\bf \tilde U}}{ {\tilde \Sigma}}{ {\bf \tilde V}}^*$.

(5) Form the rank-$k$ approximation of ${\bf A}$: $\widehat {\bf A}_k^{(q)}={\hat{\bf U}}_k{\hat \Sigma}_k\hat{\bf V}_k^*$,  where
$${\hat{\bf U}}_k={\bf Q}{\bf \tilde U}(:,1:k),\qquad {\hat \Sigma}_k={\tilde \Sigma}(1:k,1:k),\qquad \hat{\bf V}_k={\tilde {\bf V}}(:, 1:k).$$
\end{algorithm}

 To implement   Algorithm \ref{alg3.1} efficiently,   we recommend  fast  structure-preserving  quaternion Householder QR \cite{jwzc,lwz2} and QSVD algorithms \cite{wl,lwz2}. Based on structure-preserving properties (\ref{2.3}) of  the real counterpart of a quaternion matrix,
the essence  of fast structure-preserving algorithm is to store
  the four parts $Q_0, Q_1,  Q_2, Q_3$ of a quaternion matrix ${\bf Q}$ only. When the left (or right) quaternion  matrix transformation ${\bf T}_l$ (or ${\bf T}_r$)  is applied on ${\bf Q}$, it is equivalent to implementing the real matrix multiplication $\Upsilon_{{\bf T}_l}{\Upsilon_{\bf Q}}$ (or ${\Upsilon_{\bf Q}}\Upsilon_{{\bf T}_r}$).
    In order to reduce the computational cost, only the first block column (or row)  of  ${\Upsilon_{\bf Q}}$  is updated and stored. Other blocks in the updated matrix
    are not explicitly stored and formed, and they can be determined according to the real symmetry structure.
For example, in Step 2 of Algorithm \ref{alg3.1}, the four parts of quaternion matrices ${\bf Y}_0$, $\hat {\bf Y}_i$ and ${\bf Y}_i$ can be found from the computations of matrices
$$
({\bf Y}_0)_{\rm c}=\Upsilon_{\bf A}{\bf \Omega}_{\rm c},\quad \big(\hat {\bf Y}_i\big)_{\rm c}=\Upsilon_{\bf A}^T({\bf Y}_{i-1})_{\rm c},\quad
\big({\bf Y}_i\big)_{\rm c}=\Upsilon_{\bf A}\big(\hat {\bf Y}_i\big)_{\rm c},
$$
respectively, and in Step 3, the four parts of quaternion matrix  ${\bf B}$ can be found from ${\bf B}_{\rm c}=\Upsilon_{\bf Q}^T{\bf A}_{\rm c}$.
Note that the computations of $({\bf Y}_0)_{\rm c}:=\Upsilon_{\bf A}{\bf \Omega}_{\rm c}$ and the quaternion matrix multiplication ${\bf Y}_0={\bf A}{\bf \Omega}$ have the same real flops, while the former utilizes  BLAS-3 based    matrix-matrix operations better, and hence leads to   efficient computations.

Once ${\bf Y}_{q}$ is obtained,  the fast structure-preserving quaternion Householder QR algorithm \cite{lwz2}    can be applied to get the orthonormal basis matrix ${\bf Q}$.
 Here   the quaternion Householder transformation ${\bf H}$  to reduce a vector ${\bf u}\in {\mathbb Q}^{s}$ into ${\bf Hu}={\bf a}e_1$ in the QR process
 takes the form
 $$
{\bf H}=I_s-2{\bf vv}^*,\quad \mbox{with}\quad {\bf v}={{\bf u}-{\bf a}e_1\over \|{\bf u}-{\bf a}e_1\|_2},\quad  {\bf a}=\left\{\begin{array}{ll}-{{\bf u}_1\over |{\bf u}_1|}\|{\bf u}\|_2,& {\bf u}_1\not=0,\\
-\|{\bf u}\|_2,& \mbox{otherwise},\end{array}\right.
$$
where ${e}_1$ is the first column of the identity matrix $I_s$.

After computing ${\bf B}={\bf Q}^*{\bf A}$ in Step 3,
the structure-preserving QSVD  \cite{wl} of ${\bf B}$  first factorizes ${\bf B}$ into a real bidiagonal matrix $\tilde B$ \cite{lwz2}, with
the help of Golub and Reinsch's idea \cite{gr} and quaternion Householder transformation ${\bf H}_0$ \cite{lwz2}:
\begin{equation}
{\bf H}_0{\bf u}:={\rm diag}\left({ {\bf a}^*\over |{\bf a}|}, I_{s-1}\right){\bf Hu}=|{\bf a}|e_1=\|{\bf u}\|_2e_1. \label{3.2}
\end{equation}
Afterwards, the standard SVD of the real matrix $\tilde B$ completes the QSVD algorithm.

\begin{remark} The basis matrix ${\bf Q}$ in the algorithm is designed to approximate the left dominant singular subspace  of  ${\bf A}$. To get ${\bf Q}$, the structure-preserving quaternion Householder QR has better numerical stability through our numerous experiments, but with more computational cost since all columns of  a unitary matrix are computed.
 Structure-preserving quaternion modified Gram-Schmidt (QMGS) \cite[Chp. 2.4.3]{wl}
   is an economical alternative for getting the thin orthonormal factor ${\bf Q}$, but might lose the accuracy during the orthogonalization process when the input matrix has relatively small singular values.  However, when we are dealing with low-rank approximation of a large input matrix, only a small number of dominant SVD triplets are taken into account, and QMGS sometimes is sufficient to get an orthonormal basis with expected accuracy (See Example \ref{e:5-2} in Section 5).
   \end{remark}

   \begin{remark}  If $\ell$ is much smaller than  $n$, i.e., ${\bf B}$ is a ``short-and-wide'' matrix, the direct application of QSVD  on ${\bf B}$  might lead to large computational cost.  Alternatively, we recommend implementing the   QMGS of ${\bf B}^*$ as
\begin{equation}
{\bf B}^*=\hat {\bf Q}_1\hat  {\bf R}_1,  \quad \mbox{ for}\qquad\hat  {\bf Q}_1\in {\mathbb Q}^{n\times \ell},\qquad \hat  {\bf R}_1\in {\mathbb Q}^{\ell \times \ell},\label{pre-1}
\end{equation}
and then computing the QSVD of the  $\ell\times \ell$  quaternion  matrix $\hat  {\bf R}_1$ as ${\hat {\bf R}}_1=  {\hat  {\bf T}}_1{\hat  { S}}_1{\hat  {\bf Z}}_1^*$, from which the QSVD of ${\bf B}$ is given by
${\bf B}={  {\bf \tilde U}}{ {\tilde \Sigma}}{ {\bf \tilde V}}^*$ for ${\bf \tilde U}={\hat {\bf Z}}_1,  {\tilde \Sigma}={\hat { S}}_1$ and ${\bf \tilde V}={\hat {\bf Q}_1}  \hat {{\bf T}}_1$. We call the corresponding method the preconditioned randomized   QSVD (\textsf{prandsvdQ}).
\end{remark}

 \begin{remark} \label{r:3-3}  If ${\bf A}$ is Hermitian,  it can be approximated as \cite[(5.13)]{hmt}:
  \begin{equation}
 {\bf A}\approx {\bf Q}{\bf Q}^*{\bf AQ}{\bf Q}^*.\label{herm}
 \end{equation}
  Then we  form the matrix ${\bf B}={\bf Q}^*{\bf AQ}$, and use the structure-preserving {\sf eigQ}
 algorithm in \cite{jwl} to  compute the eigen-decomposition of ${\bf B}$. The corresponding algorithm  is referred to as the  {\sf randeigQ}  algorithm in the  context.
 \end{remark}

 Note that both {\sf randeigQ} and {\sf prandsvdQ}  reduce a large $n\times n$ problem into  a smaller  $\ell\times \ell$ problem. The essence of {\sf randeigQ} computes
 the eigen-decomposition of a Hermitian matrix ${\bf Q}^*{\bf A\bf Q}$, while the {\sf prandsvdQ} algorithm of ${\bf A}$ computes the QSVD of $\hat {\bf R}_1=\hat {\bf Q}_1^*{\bf AQ}$.
 For large problems with $\ell\ll n$, the cost of the two randomized algorithms is dominated by the quaternion QR procedure for getting ${\bf Q}$ and $\hat {\bf Q}_1$,
 and {\sf prandsvdQ} will cost more CPU time for the extra computation of $\hat {\bf Q}_1$, but might be more accurate in estimating the eigenvalues of ${\bf A}$.
That is because the columns of $\hat {\bf Q}_1$  span the range   space ${\cal R}({\bf AQ})$, and it is exactly  ${\cal R}({\bf A^2\Omega})$, while ${\bf Q}$ is the low-rank basis of ${\cal R}({\bf A}{\bf \Omega})$, therefore  ${\cal R}({\bf\hat Q}_1)$ might have a better approximation of the left dominant singular subspace than ${\cal R}({\bf Q})$.
   We will compare the numerical behaviors of the two algorithms in Section 5.

 For the error approximation of {\sf randeigQ}, if for some parameter $\varepsilon$,
 $ \|(I_m-{\bf QQ}^*){\bf A}\|_a\le \varepsilon$, then by \cite[(5.10)]{hmt}, the  error of
 approximating ${\bf A}$ is given by  $ \|{\bf A}-{\bf QQ}^*{\bf A}{\bf QQ}^*\|_a\le 2\varepsilon,$
 where $\varepsilon$ will be evaluated in next section.

 \begin{remark}
When the power scheme is not used in Algorithm \ref{alg3.1} (i.e. $q=0$), note that  the input matrix ${\bf A}$ in  Algorithm \ref{alg3.1} is revisited.
  However, in some circumstance, the matrix is too large to be stored. Using a similar technique to \cite{cw}, we develop a method that
requires just one pass over the matrix.
For the input Hermitian  matrix  ${\bf A}$,   according to  (\ref{herm}) and ${\bf B}={\bf Q}^*{\bf AQ}$, the sample matrix
$${\bf Y}={\bf A\Omega}\approx {\bf Q}{\bf Q}^*{\bf AQ}{\bf Q}^*{\bf \Omega}={\bf Q}{\bf B}{\bf Q}^*{\bf \Omega},$$
and  the approximation of the matrix ${\bf B}$ could be obtained by solving ${\bf BQ}^*{\bf \Omega}\approx {\bf Q}^*{\bf Y}.$

If ${\bf A}$ is not Hermitian, analogue to \cite[(5.14)-(5.15)]{hmt},   the single-pass algorithm can be  constructed based on the relation
${\bf A}\approx {\bf Q}{\bf Q}^*{\bf A}{\bf \tilde Q}{\bf\tilde Q}^*$, where $ {\bf \tilde Q}$ is the low-rank basis of ${\cal R}({\bf A}^*)$ by applying ${\bf A}^*$ on a random test matrix ${\bf \tilde \Omega} $.  The matrix ${\bf B}={\bf Q}^*{\bf A}{\bf \tilde Q}$ can be approximated by  finding a minimum-residual
solution to the system of relations  ${\bf B\tilde Q^*\Omega}={\bf Q}^*{\bf Y}$, ${\bf B^*Q^*\tilde \Omega}={\bf \tilde Q}^*{\bf \tilde Y}$ for ${\bf Y}={\bf A\Omega}$ and ${\bf \tilde Y}={\bf A}^*{\bf \tilde \Omega}$.

\end{remark}

\section{Error analysis}

The error analysis of Algorithm \ref{3.1} consists of two parts,  including the   expected values of approximation errors  $\|(I-{\bf QQ}^*){\bf A}\|_a=
\|\widehat {\bf A}_{k+p}^{(q)}-{\bf A}\|_a$ in spectral or Frobenius norm, and the probability bounds of a large deviation as well.
The argument relies on
special statistical properties of quaternion test matrix ${\bf \Omega}$. Specially, we need to evaluate the  Frobenius and spectral  norms of  ${\bf \Omega}$ and ${\bf \Omega}^\dag$.

Our theories are established based on the framework of \cite{hmt}.  To start the analysis, we require to use the information of quaternion normal distributions,  chi-squared and Wishart distributions. Some of results are provided in the literature, e.g. \cite{lo,lsm}, while  some other information needs a rather lengthy deduction.
In Section \ref{sec:4-1},  we first summarize the main results in Theorems \ref{t:4-12}-\ref{t:4-14} to show the properties of quaternion randomized  algorithm.
After investigating the statistical properties of quaternion distributions in Section \ref{sec:4-2}, we will give the detailed proofs of Theorems \ref{t:4-12}-\ref{t:4-14}  in Section \ref{sec:4-3}.

\subsection{Main results}\label{sec:4-1}

\begin{theorem}(Average Frobenius error of the {\sf randsvdQ} algorithm)\label{t:4-12} Let the QSVD of the $m\times n$ ($m\ge n$) quaternion matrix ${\bf A}$ be
$$
{\bf A}={\bf U}\Sigma{\bf V}^*={\bf U}\left[\begin{array}{cc} \Sigma_1&0\\ 0&\Sigma_2\end{array}\right]\left[\begin{array}{c} {\bf V}_1^*\\
{\bf V}_2^*\end{array}\right],\quad \Sigma_1\in {\mathbb R}^{k\times k},\quad {\bf V}_1\in {\mathbb Q}^{n\times k},
$$
where the singular value matrix $\Sigma={\rm diag}(\sigma_1,\sigma_2,\ldots,\sigma_n)$ with $\sigma_1\ge \sigma_2\ge\cdots \ge \sigma_{n}\ge 0$,
$k$  is the target rank. For  oversampling parameter $p\ge 1$, let $q=0$, $\ell=k+p\le n$ and the sample matrix ${\bf Y}_0={\bf A\Omega}$,
where  ${\bf \Omega}$ is an $n\times \ell$   quaternion random test matrix
  as in (\ref{3.1}), and   ${\bf \Omega}_1={\bf V}_1^*{\bf \Omega}$ is assumed to have full row rank, then the expected
approximation error for the rank-($k+p$) matrix $\widehat {\bf A}_{k+p}^{(0)}$ via the power scheme-free   {\sf randsvdQ} algorithm  satisfies
 $$
{\sf E}\|\widehat {\bf A}_{k+p}^{(0)}-{\bf A}\|_F\le \left(1+{\displaystyle 4k\over\displaystyle 4p+2}\right)^{1/2}\left(\sum\limits_{j>k}\sigma_j^2\right)^{1/2}.
 $$
 \end{theorem}

  \begin{theorem}\label{t:4-13}(Average    spectral error of the {\sf randsvdQ} algorithm) With the notations in Theorem \ref{t:4-12},   the expected spectral norm of the approximation error in the power scheme-free algorithm  satisfies
  \begin{equation}
{\sf E}\|\widehat {\bf A}_{k+p}^{(0)}-{\bf A}\|_2\le \left(1+3\sqrt{k\over 4p+2}~\right)\sigma_{k+1}+ {3{\rm e}\sqrt{4k+4p+2}\over 2p+2}\Big(\sum\limits_{j>k}\sigma_j^2\Big)^{1/2}.\label{4.13}
\end{equation}
If $q>0$ and the power scheme is used, then  for the rank-$(k+p)$ matrix $\widehat {\bf A}_{k+p}^{(q)}$, the spectral error satisfies
$$
{\sf E}\|\widehat {\bf A}_{k+p}^{(q)}-{\bf A}\|_2\le \left[\left(1+3\sqrt{k\over 4p+2}~\right)\sigma_{k+1}^{2q+1}+ {3{\rm e}\sqrt{4k+4p+2}\over 2p+2}\left(\sum\limits_{j>k}\sigma_j^{2(2q+1)}\right)^{1/2}\right]^{1/(2q+1)}.
$$


\end{theorem}

\begin{theorem}(Deviation bound for  approximation errors of the {\sf randsvdQ} algorithm)\label{t:4-14} With the notations in Theorem \ref{t:4-12}, we have the following estimate
for the Frobenius error
\begin{equation}
\|\widehat {\bf A}_{k+p}^{(0)}-{\bf A}\|_F\le \Big(1+t\sqrt{3k\over p+1}\Big)\Big(\sum\limits_{j>k}\sigma_j^2\Big)^{1/2}+ut{{\rm e}\sqrt{4k+4p+2}\over 4p+4}\sigma_{k+1},\label{4.14}
\end{equation}
except with the probability $2t^{-4p}+{\rm e}^{-u^2/2}$. For the spectral error,
\begin{equation}
\|\widehat {\bf A}_{k+p}^{(0)}-{\bf A}\|_2\le \left(1+{3t\over 2}\sqrt{3k\over p+1}+ut\eta_{k,p}\right)\sigma_{k+1}
+3t\eta_{k,p}\left(\sum\limits_{j>k}\sigma_j^2\right)^{1/2},\label{4.15}
\end{equation}
except with the probability $2t^{-4p}+{\rm e}^{-u^2/2}$, in which  $\eta_{k,p}={{\rm e}\sqrt{4k+4p+2}\over 4p+4}$.

\end{theorem}

\vskip 0.5cm

Theorems \ref{t:4-12}-\ref{t:4-14} reveal  that the performance
of the randomized algorithm depends strongly on the properties of singular values of ${\bf A}$.
When the   singular values of ${\bf A}$ have   fast decay rate,
  it is much easier to identify a good low-rank basis ${\bf Q}$ and provide acceptable error bounds.
 However, when the singular values of ${\bf A}$ decay slowly,  the  constructed
basis ${\bf Q}$ may have low accuracy, and the power scheme will increase the decay rate of  the singular values  of ${\bf C}=({\bf AA^*})^q{\bf A}$,
 and generate a better low-rank basis matrix.

\subsection {Statistical analysis of quaternion random test matrix}\label{sec:4-2}

In this  subsection, we aim to investigate    Frobenius and spectral  norms of the  quaternion test matrix ${\bf G}$ and its pseudoinverse, where
\begin{equation}
{\bf G}=G_0+G_1{\bf i}+G_2{\bf j}+G_3{\bf k}\in {\mathbb Q}^{m\times n},\quad m\le n,\label{4.00}
\end{equation}
and $G_0,\ldots, G_3$ are   standard Gaussian matrices whose entries are random and independently drawn from the normal distribution $N(0,1)$.
Note that  the norms of $\|{\bf G}^{\dag}\|_a$ for $a=2, F$ are closely related to the measure  of  $\big({\bf G}{\bf G}^*\big)^{-1}$, where the matrix ${\bf G}{\bf G}^*$ is named as
 a quaternion Wishart matrix. As a result,  we first recall some well known results about the  quaternion  normal distribution  and   Wishart distribution.

{}

\begin{definition}[\hspace{-0.02em}\cite{tf}] \label{d:4-1} Let ${\bf z}=z_0+z_1{\bf i}+z_2{\bf j}+z_3{\bf k}$ be a random  $m\times 1$ quaternion vector with zero mean. Define the quaternion covariance matrix ${\bf \Sigma}_m={\bf cov}({\bf z}, {\bf z})={\sf E}({\bf z}{\bf z}^*)$ as
\begin{eqnarray*}
{\bf \Sigma}_m&=&{\sf E}[(z_0+z_1{\bf i}+z_2{\bf j}+z_3{\bf k})(z_0^T-z_1^T{\bf i}-z_2^T{\bf j}-z_3^T{\bf k})]\nonumber\\
&=&\Sigma_{00}+\Sigma_{11}+\Sigma_{22}+\Sigma_{33}+(-\Sigma_{01}+\Sigma_{10}-\Sigma_{23}+\Sigma_{32}){\bf i}\nonumber\\
&{}&+(-\Sigma_{02}+\Sigma_{13}+\Sigma_{20}-\Sigma_{31}){\bf j}+
(-\Sigma_{03}+\Sigma_{30}-\Sigma_{12}+\Sigma_{21}){\bf k},
\end{eqnarray*}
in which  $\Sigma_{ij}={\rm cov}(z_i, z_j)\in {\mathbb R}^{m\times m}$ is the real covariance of random vectors $z_i$ and $z_j$.
\end{definition}

In particular, when the four parts   $z_0, z_1, z_2, z_3$ of the quaternion vector ${\bf z}$  are real independent random vectors drawn from the normal distribution $N(0,I_m)$, then the
 quaternion random vector ${\bf z}$  follows the quaternion normal distribution  $ {\bf N}(0, 4I_m)$ law, with  the possibility density function ({\sf pdf})\cite{tf}:
 ${\sf pdf}({\bf z})=(2\pi)^{-2m}{\rm etr}(-{1\over 2}{\bf z^*z}).$
We remark  that  when  ${\bf z}\sim {\bf N}(0, 4I_m)$,  $\|{\bf z}\|_2^2$ represents
  the sum of $4m$ independent real variables and each variable follows $N(0,1)$ law. Thus by the concept of real chi-squared distribution,
   $\|{\bf z}\|_2^2$ follows real chi-squared distribution $\chi_{4m}^2$ with $4m$ degrees of  freedom.

The following lemma indicates that  the quaternion normal distribution   ${\bf N}(0, 4I_m)$  is unitarily invariant.

\begin{lemma}[\hspace{-0.02em}\cite{lsm}]\label{l:4-2} For an $m\times 1$ quaternion random vector ${\bf z}\sim {\bf N}(0, 4I_m)$,  let ${\bf y}={\bf Bz} + {\bf u}$,
where ${\bf B}$ is an $m$-by-$m$ nonsingular quaternion matrix, and ${\bf u}$ is an $m$-by-1 quaternion
vector, then  ${\bf y}\sim {\bf N}({\bf u}, 4{\bf BB}^*)$.
\end{lemma}

The  rigorous  definition of the Wishart distribution is given as follows.

\begin{definition}[\hspace{-0.02em}\cite{tf,tl}]\label{d:4-3} Let ${\bf Z}=[{\bf z}_1~{\bf z}_2~ \ldots~ {\bf z}_n]$, where ${\bf z}_1, \ldots, {\bf z}_n$ are $m\times 1$
 random  independent quaternion vectors drawn from the same distribution, i.e.,  ${\bf z}_i\sim {\bf N}(0, { \bf \Sigma})(1\le i\le n)$.  Then
 ${\bf A}={\bf Z}{\bf Z}^*\in {\mathbb Q}^{m\times m}$
 is said to have the quaternion Wishart distribution with $n$ degrees of freedom and
covariance matrix ${\bf \Sigma}$. We will write that ${\bf A}\sim {\bf W}_m(n,{\bf  \Sigma})$.
\end{definition}

Note that the matrix ${\bf \Sigma}$ could be quaternion or real.  In this paper, we are only  interested  in the real case and use the notation $\Sigma$ for a distinguishment.
The matrix ${\bf A}$ is singular when $n<m$, and  the {\sf pdf}   of ${\bf A}$ doesn't exist in this case. When $m\le n$,  the   \textsf{pdf}  \cite{lsm,tl} (See also
\cite[Theorem 4.2.1]{lo}) of ${\bf A}$  exists. Before giving the {\sf pdf}, we first recall the definitions of exterior products, which are vital for the volume element of a multivariate density function.

\begin{definition}[\hspace{-0.02em}\cite{mu,lsm}]
For any $m\times n$ real matrix $X$, let ${\rm d}X=[{\rm d}x_{ij}]$ denote  the matrix of differentials,   define the $mn$-exterior product  $\{{\rm d}X\}$ of the $mn$ distinct and free elements in $X$ as
   $\{{\rm d}X\}\equiv \mathop{\wedge}\limits_{i,j} {\rm d}x_{ij}.$
For any $m\times n$ quaternion matrix ${\bf X}=X_0+{X_1}{\bf i}+{X_2}{\bf j}+{X_3}{\bf k}$, denote ${\rm d}{\bf X}={\rm d}X_0+ {\rm d}X_1~{\bf i}+{\rm d}X_2~{\bf j}+{\rm d}X_3~ {\bf k}$, and
define $\{{\rm d}{\bf X}\}=\{{\rm d}{ X}_0\}\wedge \{{\rm d}{ X}_1\}\wedge \{{\rm d}{ X}_2\}\wedge \{{\rm d}{ X}_3\}$. If ${\bf X}$ is Hermitian, then
 $X_0$ is symmetric, while $X_2, X_3, X_4$ are skew-symmetric, and
$\{{\rm d}{\bf X}\}$ takes the form
$$
\{{\rm d}{\bf X}\}=\Big(\mathop{\wedge}\limits_{i\le j}^m{\rm d}({X_0})_{ij}\Big)\wedge \Big(\mathop{\wedge}\limits_{i<j}^m{\rm d}({X_1})_{ij}\Big)
\wedge \Big(\mathop{\wedge}\limits_{i<j}^m{\rm d}({X_2})_{ij}\Big)\wedge \Big(\mathop{\wedge}\limits_{i<j}^m{\rm d}({X_3})_{ij}\Big).
$$
\end{definition}

In the definition, the exterior product of differential form in different order might differ  by a  factor $\pm 1$.
Since we are integrating exterior differential forms representing probability density functions, we ignore the sign of exterior differential forms for the sake of convenience.
Based on the notation for the exterior product, the {\sf pdf} of the quaternion Wishart matrix is given as follows.

\begin{lemma}[\hspace{-0.02em}\cite{lsm,lo}]\label{l:4-4}  Let the quaternion Wishart matrix ${\bf A}\sim  {\bf W}_m(n, {  \Sigma})$, then the {\sf pdf} of ${\bf A}$ satisfies
\begin{equation}
{\sf pdf}({\bf A})\{{\rm d}{\bf A}\}=\beta_{m,n}\left[{\rm det} ({  \Sigma})\right]^{-2n} \left[{\bf det}({\bf A})\right]^{2(n-m)+1}
{\rm etr}(-2\Sigma^{-1}{\bf A})\{{\rm d}{\bf A}\},\label{4.0}
\end{equation}
in which $\{{\rm d}{\bf A}\}$ represents the volume element of this multivariate density function, and
$$
\beta_{m,n}=2^{2mn}\pi^{-m(m-1)}\left(\prod\limits_{i=1}^m \Gamma\big(2(n-i+1)\big)\right)^{-1},
$$
with the Gamma function  $\Gamma(\cdot)$   defined by $\Gamma(x)={\displaystyle \int_0^\infty} t^{x-1}{\rm e}^{-t}dt(x>0)$.
\end{lemma}

The properties of the quaternion Wishart matrix are given as follows.

\begin{theorem}\label{t:4-5} Given ${\bf A}\sim {\bf W}_m(n, {  \Sigma})$.

(i) For ${\bf M}\in {\mathbb Q}^{k\times m}$ with ${\rm rank}({\bf M})=k$, we have
$
{\bf MAM}^*\sim {\bf W}_k(n, {\bf M}{  \Sigma}{\bf M}^*).
$

(ii) Partition
$$
{\bf A}=\left[\begin{array}{cc} {\bf A}_{11}&{\bf A}_{12}\\ {\bf A}_{12}^*& {\bf A}_{22}\end{array}\right],\quad {  \Sigma}=\left[\begin{array}{cc} {  \Sigma}_{11}&{ \Sigma}_{12}\\ { \Sigma}_{21}& {  \Sigma}_{22}\end{array}\right],
$$
in which ${\bf A}_{11} \in {\mathbb Q}^{k\times k}$, $\Sigma_{11}\in {\mathbb R}^{k\times k}$. Let ${\bf A}_{11,2}={\bf A}_{11}-{\bf A}_{12}{\bf A}_{22}^{-1}{\bf A}_{12}^*$,
${ \Sigma}_{11,2}={ \Sigma}_{11}-{  \Sigma}_{12}{  \Sigma}_{22}^{-1}{  \Sigma}_{21}$, then
$$
{\bf A}_{11,2}\sim {\bf W}_k(n-m+k, {  \Sigma}_{11,2}).\\
$$
\end{theorem}
\begin{proof}
(i)  Note that   ${\bf A}=\sum\limits_{i=1}^n {\bf z}_i{\bf z}_i^*$ with ${\bf z}_i\sim {\bf N}(0,\Sigma)$. It follows that  $\hat {\bf z}_i:=2\Sigma^{-1/2}{\bf z}_i\sim {\bf N}(0, 4I_m)$ from the definition of quaternion covariance.
By applying Lemma \ref{l:4-2}, ${\bf Mz}_i={1\over 2}({\bf M}\Sigma^{1/2}\hat {\bf z}_i)\sim {\bf N}(0, {\bf M}\Sigma{\bf M}^*)$, and hence
$$
{\bf MAM}^*=\sum\limits_{i=1}^n {\bf Mz}_i({\bf Mz}_i)^*\sim {\bf W}_k(n, {\bf M}\Sigma{\bf M}^*).
$$

(ii) Let ${\bf Z}=\left[\begin{array}{cc} I_k& 0\\-{\bf A}_{22}^{-1}{\bf A}_{12}^*&I_{m-k}\end{array}\right]$, and change the variables of ${\bf A}$ into ${\bf A}_{11,2}$,
  ${\bf B}_{12}={\bf A}_{12}$ and ${\bf B}_{22}={\bf A}_{22}$ through the following transformation
\begin{equation}
{\bf A}{\bf Z}:=\left[\begin{array}{cc} {\bf A}_{11}&{\bf A}_{12}\\ {\bf A}_{12}^*& {\bf A}_{22}\end{array}\right]\left[\begin{array}{cc} I_k& 0\\-{\bf A}_{22}^{-1}{\bf A}_{12}^*&I_{m-k}\end{array}\right]
=\left[\begin{array}{cc} {\bf A}_{11,2}&{\bf B}_{12}\\ 0& {\bf B}_{22}\end{array}\right].\label{4.xx}
\end{equation}

The quaternion matrix ${\bf Z}$ is not Hermitian, and ${\bf det}({\bf Z})$ is not well defined. In order to express ${\bf det}({\bf A})$ in terms of ${\bf det}({\bf A}_{11,2})$ and ${\bf det}({\bf B}_{22})$, we consider the transformation ${\bf Z}^*{\bf A}{\bf Z}$ to get
${\bf Z}^*{\bf A}{\bf Z}={\rm diag}({\bf A}_{11,2}, {\bf B}_{22})=:{\bf F},$
where ${\bf A}_{11,2}$ and ${\bf B}_{22}$ are Hermitian and positive definite matrices.

Take the real counter parts on both sides of ${\bf Z}^*{\bf A}{\bf Z}={\bf F}$,   the properties in (\ref{2.3}) gives
$\Upsilon_{\bf Z}^T\Upsilon_{\bf A}\Upsilon_{\bf Z}=\Upsilon_{\bf F}$ and the standard  determinant of real matrix $\Upsilon_{\bf F}$ satisfies
\begin{equation}
\det(\Upsilon_{\bf F})=\left(\det(\Upsilon_{\bf Z})\right)^2\det(\Upsilon_{\bf A}),\label{eq}
\end{equation}
where by writing the (2,1)-subblock of ${\bf Z}$ as $-{\bf A}_{22}^{-1}{\bf A}_{12}^*=\bar A_0+\bar A_1 {\bf i}+\bar A_2 {\bf j}+\bar A_3 {\bf k}$, and using the   identity matrices in block columns 2,4,6,8 of $\Upsilon_{\bf Z}$:
$$
\Upsilon_{\bf Z}=\left[\begin{array}{cc|cc|cc|cc}
I&0&0&0&0&0&0&0\\
\bar A_0& I&-\bar A_1&0&-\bar A_2&0&-\bar A_3&0\\\hline
0&0&I&0&0&0&0&0\\
\bar A_1&0&\bar A_0&I&-\bar A_3&0&\bar A_2&0\\\hline
0&0&0&0&I&0&0&0\\
\bar A_2&0&\bar A_3&0&\bar A_0&I&-\bar A_1&0\\\hline
0&0&0&0&0&0&I&0\\
\bar A_3&0&-\bar A_2&0&\bar A_1&0&\bar A_0&I
\end{array}\right]
$$
to eliminate the subblocks   $\pm \bar A_i$ to zero,  we get $\det(\Upsilon_{\bf Z})=\det(I_{4m})=1$. Thus in (\ref{eq}), $\det(\Upsilon_{\bf A})=\det(\Upsilon_{\bf F})$.
The applications of (\ref{2.4}) and  the definition (\ref{2.30}) to this equality  give
\begin{equation}
{\bf det}({\bf A})={\bf det}({\bf F})={\bf det}({\bf A}_{11,2}){\bf det}({\bf B}_{22}).  \label{4.1}
\end{equation}
For the real matrix $\Sigma$, it is obvious that
\begin{equation}
\det(\Sigma)=\det({\Sigma}_{22})\det({\Sigma}_{11,2}).\label{4.1+}
\end{equation}
By putting
$
{  C}={  \Sigma}^{-1}=\left[\begin{array}{cc} {  C}_{11}& {  C}_{12}\\ {  C}_{21}& {  C}_{22}\end{array}\right],
$
we conclude that ${  C}_{11}={ \Sigma}_{11,2}^{-1}$ and
\begin{equation}
\begin{array}{rl}
{\rm tr}({ \Sigma}^{-1}{\bf A})&={\rm tr}\left(\left[\begin{array}{cc} {  C}_{11}& {  C}_{12}\\ {  C}_{21}& {  C}_{22}\end{array}\right]\left[\begin{array}{cc}
{\bf A}_{11,2}+{\bf B}_{12}{\bf B}_{22}^{-1}{\bf B}_{12}^*& {\bf B}_{12}\\ {\bf B}_{12}^*& {\bf B}_{22}\end{array}\right]\right)\\
&={\rm tr}(C_{11}{\bf A}_{11,2})+{\rm tr}({\bf \Delta}_1)+{\rm tr}({\bf \Delta}_2)={\rm tr}({ \Sigma}_{11,2}^{-1}{\bf A}_{11,2})+{\rm tr}({\bf \Delta}_1)+{\rm tr}({\bf \Delta}_2),
\end{array}\label{4.2}
\end{equation}
where
$
{\bf \Delta}_1={ \Sigma}_{11,2}^{-1}{\bf B}_{12}{\bf B}_{22}^{-1}{\bf B}_{12}^*+{  C}_{12}{\bf B}_{12}^*,
{\bf \Delta}_2={  C}_{21}{\bf B}_{12}+{  C}_{22}{\bf B}_{22}.
$

Note that the differential of ${\bf A}_{12}{\bf A}_{22}^{-1}{\bf A}_{12}^*$ satisfies
$$
{\rm d}({\bf A}_{12}{\bf A}_{22}^{-1}{\bf A}_{12}^*)=({\rm d}{\bf A}_{12}){\bf A}_{22}^{-1}{\bf A}_{12}^*+{\bf A}_{12}({\rm d}{\bf A}_{22}^{-1}){\bf A}_{12}^*+{\bf A}_{12}
{\bf A}_{22}^{-1}({\rm d}{\bf A}_{12}^*),
$$
in which the differential  ${\rm d}({\bf A}_{22}^{-1})$ can be derived by differentiating ${\bf A}_{22}^{-1}{\bf A}_{22}=I_{m-k}$ as
 $$
({\rm d}{\bf A}_{22}^{-1}){\bf A}_{22}+{\bf A}_{22}^{-1}({\rm d}{\bf A}_{22})=0,\quad \mbox{or equivalently,}\quad {\rm d}{\bf A}_{22}^{-1}=-{\bf A}_{22}^{-1}({\rm d}{\bf A}_{22}){\bf A}_{22}^{-1}.
$$
Since the exterior products of repeated differentials are zero, we then get
$\{{\rm d}({\bf A}_{12}{\bf A}_{22}^{-1}{\bf A}_{12}^*)\}\wedge \{{\rm d}{\bf A}_{12}\}\wedge \{{\rm d}{\bf A}_{22}\}=0.$
Thus
\begin{equation}
\begin{array}{rl}
\{{\rm d}{\bf A}\}&=\{{\rm d}{\bf A}_{11}\}\wedge \{{\rm d}{\bf A}_{12}\}\wedge \{{\rm d}{\bf A}_{22}\}=\{{\rm d}({\bf A}_{11}-{\bf A}_{12}{\bf A}_{22}^{-1}{\bf A}_{12}^*)\}\wedge \{{\rm d}{\bf A}_{12}\}\wedge \{{\rm d}{\bf A}_{22}\}\\
&=
\{{\rm d} {\bf A}_{11,2}\}\wedge \{{\rm d}{\bf B}_{12}\}\wedge \{{\rm d}{\bf B}_{22}\}.
\end{array}\label{4.3}
\end{equation}

Substituting  (\ref{4.1})-(\ref{4.3}) into ${\sf pdf} ({\bf A}) \{{\rm d}{\bf A}\}$ in Lemma \ref{l:4-4}, we obtain
\begin{equation}
\begin{array}{l}
{\sf pdf}({\bf A})\{{\rm d}{\bf A}\}
=\beta_{m,n} \left([\det({ \Sigma}_{11,2})]^{-2n}\left[{\bf det}({\bf A}_{11,2})\right]^{2(n-m)+1} {\rm etr}(-2{ \Sigma}_{11,2}^{-1}{\bf A}_{11,2})\right)\\
\quad \times\left([\det({ \Sigma}_{22})]^{-2n}\left[{\bf det}({\bf B}_{22})\right]^{2(n-m)+1} {\rm etr}(-2{\bf \Delta}_1){\rm etr}(-2{\bf \Delta}_2)\right)\{{\rm d}{\bf A}_{11,2}\}\wedge \{{\rm d}{\bf B}_{12}\}\wedge \{{\rm d}{\bf B}_{22}\},
\end{array}\label{4.4}
\end{equation}
from which we see that   ${\bf A}_{11,2}$ is independent of ${\bf B}_{12}, {\bf B}_{22}$, because of the density function factors.  Notice that ${\bf A_{11,2}}$ is $k\times k$,  and
$
\left[{\bf det}({\bf A}_{11,2})\right]^{2(n-m)+1}=\left[{\bf det}({\bf A}_{11,2})\right]^{2((n-m+k)-k)+1}.
$
Moreover,  the terms in (\ref{4.4}) including ${\bf A}_{11,2}$ have close relations to the {\sf pdf} of a Wishart matrix, therefore we can find
the ${\sf pdf}$ of  ${\bf A_{11,2}}$  from  ${\sf pdf}({\bf A})$ so that ${\sf pdf}({\bf A_{11,2}})$ takes the form
$$
\begin{array}l
 \beta_{k,n-m+k} [{\rm det}({ \Sigma}_{11,2})]^{-2(n-m+k)} \left[{\bf det}({\bf A}_{11,2})\right]^{2((n-m+k)-k)+1}
{\rm etr}(-2{ \Sigma}_{11,2}^{-1}{\bf A}_{11,2}),
\end{array}
$$
which means ${\bf A}_{11,2}\sim {\bf W}_k(n-m+k,{ \Sigma}_{11,2})$. The remaining terms in (\ref{4.4}) correspond to the joint {\sf pdf} of
${\bf B}_{12}$, ${\bf B}_{22}$, whose distributions will not be considered here.
\end{proof}

Theorem \ref{t:4-5} includes the properties  of a real Wishart matrix \cite[Theorems 3.2.5 and 3.2.10]{mu} as special cases.
With Theorem \ref{t:4-5},  the expectation of $\|{\bf G}^\dag \|_F^2$ is deduced in the following theorem.

\begin{theorem} \label{t:4-6} Let the  quaternion random  matrix ${\bf G}\in {\mathbb Q}^{m\times n}$($m\le n$)   be given by (\ref{4.00}).
Then the expectation of $\|{\bf G}^\dag \|_F^2$ satisfies
$$
{\sf E}\|{\bf G}^\dag \|_F^2={\displaystyle m\over 4(n-m)+2}.
$$
\end{theorem}
\begin{proof} It is obvious that each column in ${\bf G}$ follows  ${\bf N}(0, 4I_m)$ and
\begin{equation}
{\sf E}\|{\bf G}^\dag \|_F^2= {\sf E}\left( {\rm tr} \left[({\bf GG}^*)^{-1}\right]\right)={\sf E}\sum\limits_{i=1}^m(e_i^T{\bf A}^{-1}e_i)=\sum\limits_{i=1}^m{\sf E}(e_i^T{\bf A}^{-1}e_i),
\label{4.5}
\end{equation}
where $e_i$ is the $i$-th column of the identity matrix $I_m$, and ${\bf A}={\bf G}{\bf G}^*\sim {\bf W}_m(n, 4I_m)$.

 For each fixed $i$, let $\Pi_{1,i}$ be the  permutation matrix obtained by interchanging columns $1, i$ in the $m\times m$ identity matrix, and denote
 $
 {\bf C}=\Pi_{1,i}^T{\bf A}\Pi_{1,i}=\left[\begin{array}{cc} {\bf C}_{11}& {\bf C}_{12}\\ {\bf C}_{21}& {\bf C}_{22}\end{array}\right]$ with ${\bf C}_{11}\in {\mathbb Q}^{1\times 1},
 $
 then  ${\bf C}\sim {\bf W}_m(n, 4I_m)$ by Theorem \ref{t:4-5}(i). Moreover,
 $
\left(e_i^T{\bf A}^{-1}e_i\right)^{-1}=\left(e_1^T{\bf C}^{-1}e_1\right)^{-1}={\bf C}_{11}-{\bf C}_{12}{\bf C}_{22}^{-1}{\bf C}_{21}.
$

According to  Theorem \ref{t:4-5}(ii),    $\left(e_i^T{\bf A}^{-1}e_i\right)^{-1}\sim  {\bf W}_1(n-m+1, 4),$
indicating that there exists an $(n-m+1)$-dimensional  quaternion column vector ${\bf z}\sim {\bf N}(0, 4I_{n-m+1})$  satisfying
\begin{equation}
\left(e_i^T{\bf A}^{-1}e_i\right)^{-1}=\|{\bf z}\|_2^2\sim \chi_{4(n-m+1)}^2.\label{40}
\end{equation}
By the expectation of the inverted chi-squared distribution in \cite[Proposition A.8]{hmt}, we know that
$$
{\sf E}\left(e_i^T{\bf A}^{-1}e_i\right)={\sf E}{\displaystyle 1\over\displaystyle \chi_{4(n-m+1)}^2}={\displaystyle 1\over\displaystyle 4(n-m)+2}.
$$
 The assertion in the theorem then follows.
\end{proof}

\vskip 0.2cm

The  theorem below provides a  bound on the probability of a large
deviation above the mean.

\begin{theorem}\label{t:4-7}
 Let the quaternion random  matrix ${\bf G}\in {\mathbb Q}^{m\times n}$ with $n-m\ge 1$   be given by (\ref{4.00}). Then for each $t\ge 1$,
\begin{equation}
{\sf P}\left\{\|{\bf G}^\dag\|_F^2>{3m\over 4(n-m+1)}t\right\}\le t^{-2(n-m)}.\label{F1}
\end{equation}
\end{theorem}
\begin{proof}
According to (\ref{4.5})--(\ref{40}),  $Z=\|{\bf G}^\dag\|_F^2=\sum\limits_{i=1}^m X_i$ with $X_i=e_i^T{\bf A}^{-1}e_i$ and $X_i^{-1}\sim \chi^2_{4(n-m+1)}$.  Let $q=2(n-m)$ and when $n-m\ge 1$,
the result in \cite[Lemma A.9]{hmt} ensures that
$
\|X_i\|_{L^q}:=\left[{\sf E}(|X_i|^q)\right]^{1/q}< {\displaystyle 3\over \displaystyle 4(n-m+1)}.
$
Using the  triangle inequality for the $L^q$-norm, we obtain
$$
\|Z\|_{L^q}\le \sum\limits_{i=1}^m \|X_i\|_{L^q}< {\displaystyle 3m\over \displaystyle 4(n-m+1)}=:\gamma.
$$
With Markov's inequality,
$
{\sf P}\left\{ Z\ge \gamma t\right\}={\sf P}\left\{ Z^q\ge \gamma^qt^q\right\}\le  {\displaystyle {\sf E}(Z^q)\over \displaystyle \gamma^qt^{q}}<t^{-q}=t^{-2(n-m)},
$
leading to the desired result.
\end{proof}

\vskip 0.1cm

We now turn to the estimate  of $\|{\bf G}^\dag \|_2$. Note that $\|{\bf G}^\dag \|_2=\left(\lambda_{\rm min}({\bf A})\right)^{-1/2}$, where $\lambda_{\rm min}({\bf A})$ denotes the
 smallest eigenvalue of ${\bf A}$. We therefore need to study the {\sf pdf} of the smallest eigenvalue of ${\bf A}$,  based on the following lemma and a frame work in \cite{cd} for discussing   the  eigenvalues   of a real Wishart matrix.

\begin{lemma}[\hspace{-0.02em}\cite{lsm}]\label{lem2.3} Let the quaternion Wishart matrix ${\bf A}\sim  {\bf W}_m(n, I_m)$,  then the {\sf pdf}  for the eigenvalues $\lambda_1\ge \lambda_2\ge \cdots\ge  \lambda_m>0$ of ${\bf A}$
is given by
$$
\begin{array}l
f(\lambda_1,\lambda_2,\cdots, \lambda_m)=
K_{m,n}{\displaystyle \prod_{i=1}^m\lambda_i^{2(n-m)+1}\prod_{i<j}^m}(\lambda_i-\lambda_j)^4~
{\rm e}^{-2\sum_{i=1}^m\lambda_i},
\end{array}
$$
where
$
K_{m,n}^{-1}=2^{-2mn} \pi^{2m}\prod\limits_{i=1}^m \Gamma\Big(2(n-i+1)\Big)\Gamma\Big(2(m-i+1)\Big).
$
\end{lemma}

 The following lemma gives the lower and upper bounds of  the {\sf pdf} of   $\lambda_{\rm min}({\bf A})$.

\begin{lemma}\label{l:4-9}
Let the quaternion Wishart matrix ${\bf A}\sim  {\bf W}_m(n, I_m)$, and $f_{\lambda_{\rm min}}(\lambda)$ denote the  {\sf pdf} of the smallest eigenvalue of quaternion Wishart matrix  ${\bf A}$, then
$f_{\lambda_{\rm min}}(\lambda)$ satisfies
\begin{equation}
L_{m,n}{\rm e}^{-2m\lambda }\lambda^{2(n-m)+1}\le f_{\lambda_{\rm min}}(\lambda)\le L_{m,n}{\rm e}^{-{2\lambda}}\lambda^{2(n-m)+1},\label{lem2.4}
\end{equation}
where
\begin{equation}
L_{m,n}={\displaystyle 2^{2(n-m+1)}\pi^{-2}  \Gamma(2n+2) \over \displaystyle \Gamma(2n-2m+4)\Gamma(2n-2m+2)\Gamma(2m)}.  \label{e:Lmn}
\end{equation}
\end{lemma}

\begin{proof} For $\lambda\ge 0$, let $R_{m-1}(\lambda)=\{(\lambda_1,\lambda_2,\ldots, \lambda_{m-1}):\lambda_1\ge \cdots\ge \lambda_{m-1}\ge \lambda\}\subseteq {\mathbb R}^{1\times (m-1)}$. From the {\sf pdf} of
the eigenvalues of ${\bf A}$ in Lemma \ref{lem2.3}, we have
$$
\begin{array}{ll}
f_{\lambda_{\rm min}}(\lambda)&={\displaystyle\int_{R_{m-1}(\lambda)}}f(\lambda_1,\lambda_2,\cdots, \lambda_{m-1},\lambda){\rm d}\lambda _1{\rm d}\lambda _2\cdots {\rm d}\lambda _{m-1}\\
&=K_{m,n}{\rm e}^{-2\lambda}\lambda^{2(n-m)+1}{\displaystyle\int_{R_{m-1}(\lambda)}}{\rm e}^{-2\sum_{i=1}^{m-1}\lambda_i}~{\displaystyle\prod\limits_{i=1}^{m-1}\lambda_i^{2(n-m)+1}}\\
&\qquad\qquad {\displaystyle\prod_{i=1}^{m-1}}(\lambda_i-\lambda)^4{\displaystyle\prod_{i=1}^{m-2}\prod_{j=i+1}^{m-1}}(\lambda_i-\lambda_j)^4{\rm d}\lambda _1{\rm d}\lambda _2\cdots {\rm d}\lambda _{m-1}.
\end{array}
$$
By the inequality $(\lambda_i-\lambda)^4\le \lambda_i^4$,  we find that
$$
\begin{array}{ll}
f_{\lambda_{\rm min}}(\lambda)
&\le K_{m,n}{\rm e}^{-2\lambda}\lambda^{2(n-m)+1}{\displaystyle\int_{R_{m-1}(0)}}{\rm e}^{-2\sum_{i=1}^{m-1}\lambda_i}
{\displaystyle\prod\limits_{i=1}^{m-1}\lambda_i^{2(n-m)+5}}\\
&\qquad\qquad {\displaystyle\prod_{i=1}^{m-2}\prod_{j=i+1}^{m-1}}(\lambda_i-\lambda_j)^4{\rm d}\lambda_1{\rm d}\lambda_2\cdots {\rm d}\lambda_{m-1}\\
&=: K_{m,n}{\rm e}^{-2\lambda}\lambda^{2(n-m)+1}C_{m,n}.
\end{array}
$$

For the lower bound, set $\mu_i=\lambda_i-\lambda$($i=1,\ldots,m-1$), then $\mu_1\ge \mu_2\ge \cdots\ge \mu_{m-1}\ge 0$, and
$$
\begin{array}{rl}
f_{\lambda_{\rm min}}(\lambda)&=K_{m,n}{\rm e}^{-2m\lambda}\lambda^{2(n-m)+1}{\displaystyle\int_{R_{m-1}(0)}}{\rm e}^{-2\sum_{i=1}^{m-1}\mu_i}
{\displaystyle\prod\limits_{i=1}^{m-1}(\mu_i+\lambda)^{2(n-m)+1}} \\
&{}\qquad\qquad {\displaystyle\prod_{i=1}^{m-1}}\mu_i^4{\displaystyle\prod_{i=1}^{m-2}\prod_{j=i+1}^{m-1}}(\mu_i-\mu_j)^4{\rm d}\mu_1{\rm d}\mu_2\cdots {\rm d}\mu_{m-1}\nonumber\\
&\ge K_{m,n}{\rm e}^{-2m\lambda}\lambda^{2(n-m)+1}{\displaystyle\int_{R_{m-1}(0)}}{\rm e}^{-2\sum_{i=1}^{m-1}\mu_i}
{\displaystyle\prod\limits_{i=1}^{m-1}\mu_i^{2(n-m)+5}}\\
&{}\qquad\qquad {\displaystyle{\displaystyle\prod_{i=1}^{m-2}\prod_{j=i+1}^{m-1}}(\mu_i-\mu_j)^4{\rm d}\mu_1{\rm d}\mu_2\cdots {\rm d}\mu_{m-1}}\\
&= K_{m,n}{\rm e}^{-2m\lambda}\lambda^{2(n-m)+1}C_{m,n}.
\end{array}
$$

Note that $f(\lambda_1,\ldots,\lambda_m)$ is a probability density function, therefore by the expression of $K_{m,n}$ in Lemma \ref{lem2.3},
$$
{\displaystyle\int_{ R_m(0)}}{\rm e}^{-2\sum_{i=1}^{m}\lambda_i}
{\displaystyle\prod\limits_{i=1}^{m}\lambda_i^{2(n-m)+1}}{\displaystyle\prod_{i=1}^{m-1}}{\displaystyle\prod_{j=i+1}^{m}}(\lambda_i-\lambda_j)^4{\rm d}\lambda _1{\rm d}\lambda _2\cdots {\rm d}\lambda _{m}=K_{m,n}^{-1}.
$$
It then follows that  $C_{m,n}=K_{m-1,n+1}^{-1}$ and hence the inequality \eqref{lem2.4} holds,
where  $L_{m,n}={K_{m,n}\over K_{m-1,n+1}}$ and it takes the form
\eqref{e:Lmn} by Theorem \ref{t:4-5}(i).
The assertion in the lemma then follows.
\end{proof}

\begin{theorem}\label{t:4-10}  Let ${\bf G}\in {\mathbb Q}^{m\times n}$ be given by (\ref{4.00}).
Then
\begin{equation}
{\sf P}\big\{\|{\bf G}^\dag \|_2>{\displaystyle {\rm e}\sqrt{4n+2}\over\displaystyle  4(n-m+1)}t\big\}\le {\displaystyle \pi^{-3}\over\displaystyle  4(n-m+1)(2n-2m+3)}t^{-4({n-m+1})},\label{4.6}
\end{equation}
and
$
{\sf E}\|{\bf G}^\dag\|_2\le  {{\rm e}\sqrt{4n+2}\over 2n-2m+2}.
$
\end{theorem}

\begin{proof}
 Note that the columns of  ${\bf G}$ follow  ${\bf N}(0, 4I_m)$ law, therefore according to  Theorem \ref{t:4-5}(i),
 ${\bf A}={1\over 4}{\bf GG}^*\sim  {\bf W}_m(n, I_m)$.

 Assume  that $\lambda_{\rm min}$ is the smallest eigenvalue  of ${\bf A}$.  By Lemma \ref{l:4-9}, we know that
$$
\begin{array}{rl}
{\sf P}\big\{\lambda_{\rm min}< \gamma\big\}&={\displaystyle \int_{0}^\gamma} f_{\lambda_{\rm min}}(t){\rm d}t\le L_{m,n}{\displaystyle \int_{0}^\gamma}t^{2(n-m)+1}{\rm d}t\\\\
&\le{\displaystyle 2^{2(n-m+1)}\pi^{-2}(2n+1)^{2(n-m+1)}\Gamma(2m)\over\displaystyle \Gamma(2n-2m+4)\Gamma(2n-2m+2)\Gamma(2m) }{\displaystyle \gamma^{2n-2m+2}\over\displaystyle 2n-2m+2}\\
&={\displaystyle \pi^{-2}(4n+2)^{2n-2m+2}\over\displaystyle (2n-2m+3)[\Gamma(2n-2m+3)]^2}\gamma^{2n-2m+2}\\
&\approx  {\displaystyle \pi^{-3}\over\displaystyle 4(n-m+1)(2n-2m+3)}\Big[{\displaystyle {\rm e}\sqrt{4n+2}\over\displaystyle  2n-2m+2}\Big]^{2(2n-2m+2)}\gamma^{2n-2m+2}\\
&=: C\gamma^{2n-2m+2},
\end{array}
$$
where we have used the  Stirling's approximation formula  $\Gamma(n+1)=n!\approx \sqrt{2\pi n}\big({n\over {\rm e}}\big)^n$. Thus
$$
{\sf P}\big\{\|{\bf G}^\dag \|_2>\tau \big\}={\sf P}\big\{\lambda_{\rm min}<{1\over 4}\tau^{-2}\big\}\le
\bar C\tau^{-2({2n-2m+2})},
$$
for $\bar C=C/4^{2n-2m+2}$. The estimate in (\ref{4.6}) is derived.

To estimate  ${\sf E}\|{\bf G}^\dag \|_2$, set  $\ell=2(n-m+1)$, then  for any $a\ge 0$,
$$
{\sf E}\|{\bf G}^\dag\|_2={\displaystyle \int_{0}^{+\infty}} {\sf P}\big\{\|{\bf G}^\dag \|_2>\tau\big\}{\rm d}\tau\le a+{\displaystyle \int_{a}^{+\infty}} {\sf P}\big\{\|{\bf G}^\dag \|_2>\tau\big\}{\rm d}\tau\le a+{\displaystyle {\bar C}a^{1-2\ell}\over \displaystyle 2\ell -1},
$$
where the right-hand side is minimized for $a={\bar C}^{1/(2\ell)}=2^{-1}C^{1/(2\ell)}$. Then
$$
{\sf E}\|{\bf G}^\dag\|_2\le (1+{1\over 2\ell-1}){\bar C}^{1/(2\ell)}\le 2{\bar C}^{1/(2\ell)}\le  {\displaystyle {\rm e}\sqrt{4n+2}\over\displaystyle 2n-2m+2}.
$$
The assertion for ${\sf E}\|{\bf G}^\dag \|_2$ then follows.
\end{proof}

The spectral or Frobenius norm of ${\bf G}$ is also vital for our error analysis. For the real Gaussian matrix $\tilde G$, the expectation of spectral or Frobenius norm of the scaled matrix $\tilde S\tilde G\tilde T$  has been proven to satisfy the following sharp bounds\cite[Proposition 10.1]{hmt}:
\begin{equation}
{\sf E}\|\tilde S\tilde G\tilde T\|_F^2=\|\tilde S\|_F^2\|\tilde T\|_F^2,\quad {\sf E}\|\tilde S\tilde G\tilde T\|_2\le \|\tilde S\|_2\|\tilde T\|_F+\|\tilde S\|_F\|\tilde T\|_2.\label{4.9}
\end{equation}
Based on above results, we present the estimates for the norms of quaternion scaled matrix  ${\bf SGT}$.

\begin{lemma}\label{l:4-11} Let ${\bf G}\in {\mathbb Q}^{m\times n}$ be given by  (\ref{4.00}), and ${\bf S}\in {\mathbb Q}^{l\times m}, {\bf T}\in {\mathbb Q}^{n\times r}$  be any two  fixed quaternion matrices,
 then
 \begin{eqnarray}
{\sf E}\|{\bf SGT}\|_F^2&=&4\|{\bf S}\|_F^2\|{\bf T}\|_F^2,\label{4.7}\\
{\sf E}\|{\bf SGT}\|_2&\le&3(\|{\bf S}\|_2\|{\bf T}\|_F+\|{\bf S}\|_F\|{\bf T}\|_2).\label{4.8}
\end{eqnarray}
 \end{lemma}

 \begin{proof} Note that the distribution of ${\bf G}$ and Frobenius norm of a matrix are both invariant under unitary transformations. As a result, without
loss of generality, we assume that ${\bf S}, {\bf T}$ are real diagonal matrices  whose diagonal entries are exactly their singular values. Write ${\bf S}=S, {\bf T}=T$,
it follows that
$$
{\sf E}\|{\bf SGT}\|_F^2={\sf E}\sum\limits_{k,j} (|{s}_{kk}{\bf g}_{kj}{t}_{jj}|)^2=\sum\limits_{k,j} |{s}_{kk}|^2|{t}_{jj}|^2{\sf E}|{\bf g}_{kj}|^2=4\|{\bf S}\|_F^2\|{\bf T}\|_F^2,
$$
where ${\sf E}|{\bf g}_{kj}|^2=4$ because the quaternion number ${\bf g}_{kj}$ follows ${\bf N}(0,4)$ law.

For the spectral norm, by the real counter part of ${\bf SGT}$, we know that
$\|{\bf SGT}\|_2=\|\Upsilon_{S}\Upsilon_{\bf G}\Upsilon_{T}\|_2$  in which  $\Upsilon_{\bf G}$ has dependent subblocks, and hence it is not a real Gaussian matrix.
  In order to apply the result in $(\ref{4.9})$ to the quaternion spectral norm estimation,
write $\Upsilon_{\bf G}$  in terms of its first block column  ${\bf G}_{\rm c}$:
 \begin{equation}
 \Upsilon_{\bf G}=[J_0{\bf G}_{\rm c}\quad J_1{\bf G}_{\rm c}\quad J_2 {\bf G}_{\rm c}\quad J_3{\bf G}_{\rm c}],\label{4.10}
 \end{equation}
 where  ${\bf G}_{\rm c}$ is a real Gaussian matrix,  $J_0=I_{4m}$ and
\begin{equation}
 \begin{array}l
 J_1=\left[\begin{array}r
 -e_2^T\\e_1^T\\e_4^T\\-e_3^T\end{array}\right]\otimes I_m,
   \quad J_2=\left[\begin{array}r
 -e_3^T\\-e_4^T\\e_1^T\\e_2^T\end{array}\right]\otimes I_m, \quad J_3=\left[\begin{array}r
 -e_4^T\\e_3^T\\-e_2^T\\e_1^T\end{array}\right]\otimes I_m,
  \end{array}\label{4.11}
 \end{equation}
and $e_i$ is the $i$-th column of the $4\times 4$ identity matrix.

Note that  for four arbitrary  real matrices ${ M}_0,\ldots, {M}_3$ with the same rows,
$$\|[{ M}_0~~ { M}_1 ~~{ M}_2~~ { M}_3]\|_2=\|\sum\limits_{i=0}^3{M}_i{ M}_i^*\|_2^{1/2}\le 2\max\limits_{0\le i\le 3}\|{ M}_i\|_2.$$
Using this inequality to evaluate the spectral norm of ${\bf SGT}$, we obtain
$$
\begin{array}{rl}
\|{\bf SGT}\|_2&=\|\Upsilon_{S}[J_0{\bf G}_{\rm c}T\quad J_1{\bf G}_{\rm c}T\quad J_2{\bf G}_{\rm c}T\quad J_3{\bf G}_{\rm c}T]\|_2\le 2\max\limits_{0\le k\le 3} \|\Upsilon_{S}J_k{\bf G}_{\rm c}T\|_2=2\|\Upsilon_{S}{\bf G}_{\rm c}T\|_2,
\end{array}
$$
where we have used the facts    $J_k^T\Upsilon_{S}J_k=\Upsilon_{S}$ and $\|\Upsilon_{S}J_k{\bf G}_{\rm c}T\|_2=\|\Upsilon_{S}{\bf G}_{\rm c}T\|_2.$

Therefore by (\ref{4.9}) and (\ref{2.5}), we have
$$
{\sf E}\|{\bf SGT}\|_2\le 2\left(\|\Upsilon_{S}\|_2\|T\|_F+\|\Upsilon_{S}\|_F\|T\|_2\right)=2\|{\bf S}\|_2\|{\bf T}\|_F+4\|{\bf S}\|_F\|{\bf T}\|_2.
$$
 By applying above estimates  to evaluate ${\sf E}\|{\bf SGT}\|_2={\sf E}\|{\bf T}^*{\bf G}^*{\bf S}^*\|_2$, we  obtain ${\sf E}\|{\bf SGT}\|_2\le  2\|{\bf S}\|_F\|{\bf T}\|_2+4\|{\bf S}\|_2\|{\bf T}\|_F.$ Take the average of the two upper bounds of ${\sf E}\|{\bf SGT}\|_2$, the assertion in (\ref{4.8}) follows.
\end{proof}

\vskip 0.3cm

\subsection{Proofs of Theorems \ref{t:4-12}-\ref{t:4-14}}\label{sec:4-3}

Throughout this   subsection, $\|\cdot\|_a$ denotes either the spectral norm or Frobenius norm.

{\bf Proof of Theorem \ref{t:4-12}.}   Let ${\bf Q}$ be the orthonormal basis for  the range of the sample matrix ${\bf Y}_0={\bf A\Omega}$.  Set ${\bf \Omega}_i={\bf V}_i^*{\bf \Omega}$ for $i=1,2$,  then  by a similar deduction to \cite[Theorem 9.1]{hmt},   the following inequality
\begin{equation}
\|\widehat {\bf A}_{k+p}^{(0)}-{\bf A}\|_a^2=\|(I_m-{\bf Q}{\bf Q}^*){\bf A}\|_a^2\le \|\Sigma_2\|_a^2+\|\Sigma_2{\bf \Omega}_2{\bf \Omega}_1^\dag\|_a^2\le \left(\|\Sigma_2\|_a+\|\Sigma_2{\bf \Omega}_2{\bf \Omega}_1^\dag\|_a\right)^2, \label{4.12}
 \end{equation}
 also holds for the quaternion case, in which   ${\bf V}^*{\bf \Omega}$ follows the ${\bf N}(0, 4I_n)$ law. By Lemma \ref{l:4-2},
${\bf \Omega}_1$, ${\bf \Omega}_2$ are disjoint submatrices of ${\bf V}^*{\bf \Omega}$ with the   $k\times (k+p)$ matrix ${\bf \Omega}_1$   of full row rank with  probability one.

By   Jensen's inequality  to (\ref{4.12}), we know that
$$
  \begin{array}{rl}
{\sf E} \|\widehat {\bf A}_{k+p}^{(0)}-{\bf A}\|_F&\le
\left({\sf E}\|\widehat {\bf A}_{k+p}^{(0)}-{\bf A}\|_F^2\right)^{1/2}\le \left(\|\Sigma_2\|_F^2+{\sf E}\|\Sigma_2{\bf \Omega}_2{\bf \Omega}_1^\dag\|_F^2\right)^{1/2},
\end{array}
 $$
 where by conditioning on the value of ${\bf \Omega}_1$ and applying (\ref{4.7}) to the scaled matrix $\Sigma_2{\bf \Omega}_2{\bf \Omega}_1^\dag$,
 $$
 {\sf E}\|\Sigma_2{\bf \Omega}_2{\bf \Omega}_1^\dag\|_F^2= {\sf E}\left({\sf E}\Big[\|\Sigma_2{\bf \Omega}_2{\bf \Omega}_1^\dag\|_F^2~|~{\bf \Omega}_1\Big]\right)
=4\|\Sigma_2\|_F^2{\sf E}\|{\bf \Omega}_1^\dag\|_F^2,
 $$
 which is exactly ${\displaystyle 4k\over\displaystyle 4p+2}\|\Sigma_2\|_F^2$ according to Theorem \ref{t:4-6}. The assertion in Theorem \ref{t:4-12} then follows.\hfill{$\square$}

{\bf Proof of Theorem \ref{t:4-13}.}  From (\ref{4.12}), it is obvious that
  $
{\sf E}\|\widehat {\bf A}_{k+p}^{(0)}-{\bf A}\|_2\le \|\Sigma_2\|_2+{\sf E}\|\Sigma_2{\bf \Omega}_2{\bf \Omega}_1^\dag\|_2,
 $
 where by conditioning on the value of ${\bf \Omega}_1$ and applying (\ref{4.8}) to the scaled matrix $\Sigma_2{\bf \Omega}_2{\bf \Omega}_1^\dag$,
 $$
 \begin{array}{rl}
 {\sf E}\|\Sigma_2{\bf \Omega}_2{\bf \Omega}_1^\dag\|_2&= {\sf E}\left({\sf E}\Big[\|\Sigma_2{\bf \Omega}_2{\bf \Omega}_1^\dag\|_2|{\bf \Omega}_1\Big]\right)
\le 3{\sf E}(\|\Sigma_2\|_2\|{\bf \Omega}_1^\dag\|_F+\|\Sigma_2\|_F\|{\bf \Omega}_1^\dag\|_2)\\
&\le 3\|\Sigma_2\|_2\left({\sf E}\|{\bf \Omega}_1^\dag\|_F^2\right)^{1/2}+3\|\Sigma_2\|_F{\sf E}\|{\bf \Omega}_1^\dag\|_2.
\end{array}
 $$
The estimate for the expectation of the error then follows from Theorems \ref{t:4-6} and \ref{t:4-10}.

For the power scheme, let $\tilde {\bf Q}$ be the orthonormal basis for the range of ${\bf Y}_q={\bf C\Omega}=\left({\bf AA^*}\right)^q{\bf A}{\bf \Omega}={\bf U}\Sigma^{2q+1}{\bf V}^*$.
By Jensen's inequality and a similar deduction to  \cite[Theorem 9.2]{hmt}, we know that
  $$
{\sf E}\|\widehat {\bf A}_{k+p}^{(q)}-{\bf A}\|_2={\sf E}\|(I_m-\tilde {\bf Q}\tilde{\bf Q}^*){\bf A}\|_2\le\left({\sf E}\|(I_m-\tilde{\bf Q}\tilde{\bf Q}^T){\bf C}\|_2\right)^{1/(2q+1)},
 $$
where $\sigma_1^{2q+1},\ldots, \sigma_n^{2q+1}$ are the singular values of ${\bf C}$. The assertion for the power scheme comes true by invoking the result in (\ref{4.13}). \hfill{$\square$}

\begin{remark} By using the relation $\sum\limits_{j>k}\sigma_j^{2q+1}\le (\min(m,n)-k)\sigma_{k+1}^{2q+1}$, the spectral error in Theorem \ref{t:4-13} is  bounded by
$
{\sf E}\|\widehat {\bf A}_{k+p}^{(q)}-{\bf A}\|_2\le \sigma_{k+1}\!\left[ 1+3\sqrt{k\over 4p+2}+ {3{\rm e}\sqrt{4k+4p+2}\over 2p+2}\sqrt{{\rm min}(m,n)-k}\right]^{1/(2q+1)}.
$
The power scheme drives the extra factor in
the error to one exponentially fast through increasing the exponent $q$, and by the time $q\sim \log (\min(m,n))$, ${\sf E}\|\widehat {\bf A}_{k+p}^{(q)}-{\bf A}\|_2\sim \sigma_{k+1}$.
\end{remark}

\vskip 0.1cm

 The analysis of deviation bounds for   approximation errors in Theorem \ref{t:4-14}
relies on the following well-known concentration result \cite[Proposition 10.3]{hmt} for functions of a real Gaussian matrix.

\begin{lemma}[\hspace{-0.02em}\cite{hmt}]\label{l:4-12}  Suppose that $h(\cdot)$ is a Lipschitz function on real matrices:
$
|h(X)-h(Y)|\le L\|X-Y\|_F$
for all
$X, Y\in {\mathbb R}^{s\times t}.$
Then for an $s\times t$ standard real Gaussian matrix $G$,
$
{\sf P}\{h(G)\ge {\sf E}h(G)+Lu\}\le {\rm e}^{-u^2/2}.
$
\end{lemma}

{\bf Proof of Theorem \ref{t:4-14}.}   For $t\ge 1$,   define the parameterized event on which the spectral and Frobenius norms of ${\bf \Omega}_1$ are both controlled:
\begin{equation}
E_t=\left\{{\bf \Omega}_1: \|{\bf \Omega}_1^\dag\|_2\le  {{\rm e}\sqrt{4k+4p+2}\over 4(p+1)}\cdot t ~~  \mbox{and} ~~  \|{\bf \Omega}_1^\dag\|_F\le \sqrt{3k\over 4p+4}\cdot t\right\}.\label{cond}
\end{equation}
By Theorems \ref{t:4-7} and \ref{t:4-10}, the probability of the complement of this event  satisfies a simple bound
$$
{\sf P}(E_t^{\rm c})\le  t^{-(4p+4)}+t^{-4p}\le 2t^{-4p},
$$
according to the estimates in (\ref{F1})-(\ref{4.6}).

Set $\bar h({\bf X})=\|\Sigma_2{\bf X}{\bf \Omega}_1^\dag\|_F$, in which  the real counter part of  an $(n-k)\times k$ quaternion matrix ${\bf X}$   can be represented on the basis of
${\bf X}_{\rm c}$ as
$
\Upsilon_{\bf X}=[J_0{\bf X}_{\rm c}~J_1{\bf X}_{\rm c}~J_2{\bf X}_{\rm c}~J_3{\bf X}_{\rm c}]
$
for $J=[J_0~ J_1~ J_2 ~J_3],$ and $J_k\in {\mathbb R}^{4(n-k)\times 4(n-k)}$ has similar structure  to the one in  (\ref{4.11}).

Owing to (\ref{2.5})-(\ref{2.6}), $\bar h({\bf X})={1\over 2}\|\Upsilon_{\Sigma_2}\Upsilon_{\bf X}\Upsilon_{{\bf \Omega}_1^\dag}\|_F$ and  we could write $\bar h({\bf X})$ as  a function of ${\bf X}_{\rm c}$ with
 $h({\bf X}_{\rm c}):=\bar h({\bf X})$. Notice that $h({\bf X}_{\rm c})$  is a Lipschitz function on real matrices:
\begin{equation}
\begin{array}{rl}
|h({\bf X}_{\rm c})-  h({\bf Y}_{\rm c})|&=\left|\|\Sigma_2{\bf X}{\bf \Omega}_1^\dag\|_F-\|\Sigma_2{\bf Y}{\bf \Omega}_1^\dag\|_F\right|\le\|\Sigma_2({\bf X}-{\bf Y}){\bf \Omega}_1^\dag\|_F \\
&  \le  \|\Sigma_2\|_2\|{\bf \Omega}_1^\dag\|_2\|{\bf X}-{\bf Y}\|_F = \|\Sigma_2\|_2\|{\bf \Omega}_1^\dag\|_2\|{\bf X}_{\rm c}-{\bf Y}_{\rm c}\|_F,
\end{array}\label{4.21}
\end{equation}
with a Lipschitz constant $L\le  \|\Sigma_2\|_2\|{\bf \Omega}_1^\dag\|_2$. With Jensen's inequality and
Lemma \ref{l:4-11}, we get
$$
{\sf E}[\bar h({\bf \Omega}_2)~|~{\bf \Omega}_1]\le \left({\sf E}[\left(\bar h({\bf \Omega}_2)\right)^2~|~{\bf \Omega}_1]\right)^{1/2}=2\|\Sigma_2\|_F\|{\bf \Omega}_1^\dag\|_F,
$$
where $\bar h({\bf \Omega}_2)=h\big(({\bf \Omega}_2)_{\rm c}\big)$, and  $({\bf \Omega}_2)_{\rm c}$ is a real Gaussian matrix. Applying Lemma \ref{l:4-12}, conditionally to  the random variable $\bar h({\bf \Omega}_2)=\|\Sigma_2{\bf \Omega}_2{\bf \Omega}_1^\dag\|_F$ gives
$$
P_{u,t}:={\sf P}\left\{\|\Sigma_2{\bf\Omega}_2{\bf \Omega}_1^\dag\|_F >2\|\Sigma_2\|_F\|{\bf \Omega}_1^\dag\|_F+ \|\Sigma_2\|_2\|{\bf \Omega}_1^\dag\|_2u~|~ E_t\right\}\le {\rm e}^{-u^2/2}.
$$
In (\ref{cond}), consider  the upper bounds associated with the event $E_t$  and  substitute them  into the above inequality, then we can get
$$
{\sf P}\left\{\|\Sigma_2{\bf\Omega}_2{\bf \Omega}_1^\dag\|_F > \sqrt{3k\over p+1}\|\Sigma_2\|_Ft+{{\rm e}\sqrt{4k+4p+2}\over 4p+4}\|\Sigma_2\|_2ut~|~ E_t\right\}\le
P_{u,t}\le {\rm e}^{-u^2/2}.
$$
Using ${\sf P}(E_t^{\rm c})\le 2t^{-4p}$ to remove the conditioning, we obtain
$$
{\sf P}\left\{\|\Sigma_2{\bf\Omega}_2{\bf \Omega}_1^\dag\|_F >\sqrt{3k\over p+1}\Big(\sum\limits_{j>k}\sigma_j^2\Big)^{1/2}t+ut{{\rm e}\sqrt{4k+4p+2}\over 4p+4}\sigma_{k+1}\right\}\le 2t^{-4p}+{\rm e}^{-u^2/2}.
$$
In terms of (\ref{4.12}), $\|\widehat {\bf A}_{k+p}^{(0)}-{\bf A}\|_F\le \|\Sigma_2\|_F+\|\Sigma_2{\bf \Omega}_2{\bf \Omega}_1^\dag\|_F$,   the desired probability
bound  in  (\ref{4.14}) follows.

For the deviation bound of the spectral error, set $\tilde h({\bf X})=\|\Sigma_2{\bf X}{\bf \Omega}_1^\dag\|_2$, and view $\tilde h({\bf X})$  as a function
of ${\bf X}_{\rm c}$, i.e. $ \check{h}({\bf X}_{\rm c})=\tilde h({\bf X})$, then
$$
|\check{h}({\bf X}_{\rm c})-\check{h}({\bf Y}_{\rm c})|\le \|\Sigma_2\|_2\|{\bf X}-{\bf Y}\|_2\|{\bf \Omega}_1^\dag\|_2\le \|\Sigma_2\|_2\|{\bf \Omega}_1^\dag\|_2\|{\bf X}-{\bf Y}\|_F
=\|\Sigma_2\|_2\|{\bf \Omega}_1^\dag\|_2\|{\bf X}_{\rm c}-{\bf Y}_{\rm c}\|_F,
$$
 from which we know that $\check{h}(\cdot)$ is also a Lipschitz function   with the Lipschitz constant $L\le   \|\Sigma_2\|_2\|{\bf \Omega}_1^\dag\|_2$.
Using the upper bound for the expectation of $\tilde h({\bf \Omega})$:
$$
{\sf E}[\tilde h({\bf \Omega}_2)~|~{\bf \Omega}_1]\le 3\left(\|\Sigma_2\|_2\|{\bf \Omega}_1^\dag\|_F+\|\Sigma_2\|_F\|{\bf \Omega}_1^\dag\|_2\right),
$$
and the concentration result in Lemma \ref{l:4-12},
 it follows that
$$
{\sf P}\left\{\|\Sigma_2{\bf\Omega}_2{\bf \Omega}_1^\dag\|_2 >3(\|\Sigma_2\|_2\|{\bf \Omega}_1^\dag\|_F+ \|\Sigma_2\|_F\|{\bf \Omega}_1^\dag\|_2)+ \|\Sigma_2\|_2\|{\bf \Omega}_1^\dag\|_2u~|~ E_t\right\}\le {\rm e}^{-u^2/2}.
$$
The bound in (\ref{4.15}) could be derived from (\ref{4.12})-(\ref{cond}) with a similar technique. \hfill{$\square$}

\begin{corollary}(Simple deviation bound for the spectral error of  power scheme-free algorithm)\label{c:4-14} With the notations in Theorem \ref{t:4-12}, we have the simple
upper bound
\begin{equation}
\|\widehat {\bf A}_{k+p}^{(0)}-{\bf A}\|_2\le \left(1+18\sqrt{1+{k\over p+1}}\right)\sigma_{k+1}
+{6\sqrt{4k+4p+2}\over p+1}\left(\sum\limits_{j>k}\sigma_j^2\right)^{1/2},\label{4.17}
\end{equation}
except with the probability $3{\rm e}^{-4p}$.
\end{corollary}

\begin{proof}
 Taking $u=2\sqrt{2p}, t={\rm e}$ in Theorem \ref{t:4-14} leads to
 $$
 \begin{array}{rl}
 \|\widehat {\bf A}_{k+p}^{(0)}-{\bf A}\|_2&\le \left(1+{3{\rm e}\over 2}\sqrt{3k\over p+1}+{2\sqrt{2p}{\rm e}^2\over 2\sqrt{p+1}}\sqrt{1+{k\over p+1}}\right)\sigma_{k+1}+{3{\rm e}^2\sqrt{4k+4p+2}\over 4p+4}\left(\sum\limits_{j>k}\sigma_j^2\right)^{1/2}\\
 &\le \left(1+({3\sqrt{3}{\rm e}\over 2}+\sqrt{2}{\rm e}^2)\sqrt{1+{k\over p+1}}\right)\sigma_{k+1}+{6\sqrt{4k+4p+2}\over p+1}\left(\sum\limits_{j>k}\sigma_j^2\right)^{1/2},
 \end{array}
  $$
 from which the desired upper bound follows.
\end{proof}

\section{Numerical examples}

\setcounter{figure}{0}

In this section, we give five examples to test the features of randomized QSVD algorithms. The following numerical examples are performed via MATLAB
  with machine precision $u = 2.22e-16$ in a laptop with Intel Core (TM) i5-8250U CPU  @ 1.80GHz  and the memory is 8 GB.
  Algorithms such as quaternion QR, QSVD are coded  based on the structure-preserving scheme.

\begin{example}\label{e:5-1} {\rm In this example, we test the rationality of estimated bounds  for approximation errors $\|\widehat {\bf A}_{k+p}^{(q)}-{\bf A}\|_a$.  To this end, we construct an  $m\times n(m\ge n)$  quaternion random matrix ${\bf A}$ as
$
{\bf A}={\bf U}\left[{\Sigma_1\atop 0}\right] {\bf V}^*,
$
where ${\bf U}, {\bf V}$ are quaternion Householder matrices taking the form ${\bf U}=I_m-2{\bf uu}^*, {\bf V}=I_n-2{\bf vv}^*$,  ${\bf u, v}$ are quaternion unit vectors, and
$\Sigma_1={\rm diag }(\sigma_1,\ldots, \sigma_n)$  is the  real $n\times n$ diagonal matrix. Consider singular values with different decay rate as

(1) $\sigma_1=1,  \sigma_{i+1}/\sigma_{i}=0.9$  for $i=1\ldots, n-1$  or

(2) $\sigma_1=1,  \sigma_{i+1}/\sigma_{i}=0.1$ for $i=1\ldots, n-1$, \\
where in case (1), the smallest singular value is $\sigma_{80}\approx   2.18\cdot 10^{-4}$, while in case (2), for the threshold $\theta=10^{-15}$, the numerical rank of the matrix is  16.

For each case with different values of $k, p$, we run Algorithm \ref{alg3.1} with $q=0$ for 1000 times, and plot the histograms for exact values of $\|\widehat {\bf A}_{k+p}^{(0)}-{\bf A}\|_a$ with $a=2, F$.
Below each histogram, the upper bounds of the  errors are listed, where we take $p=4$ for all cases, and   the
 bound $\eta_a^e$ for average   errors is estimated  via Theorems \ref{t:4-12} and \ref{t:4-13}, while the bound $\eta_a^d$ for deviation errors is based on  (\ref{4.14}) and (\ref{4.17}), respectively, in which $u=2\sqrt{2p}, t={\rm e}$.  For $p\ge 4$, the bounds hold with probability $99.99\%$.

\begin{figure}
\centering
\includegraphics[width=0.8\textwidth,height=8cm]{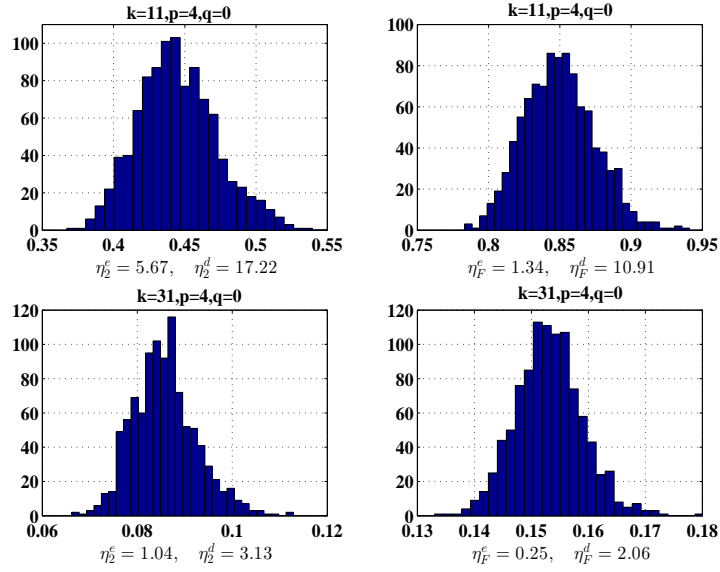}
\caption{\it Approximation errors and upper bounds  for a  $100\times 80$ matrix whose singular values decay very slowly (decay rate: 0.9). The left figures are for the estimates of  spectral errors, while the right ones correspond to the Frobenius errors. }\label{Fig5.1}
 \end{figure}

\begin{figure}
\centering
\includegraphics[width=0.8\textwidth,height=8cm]{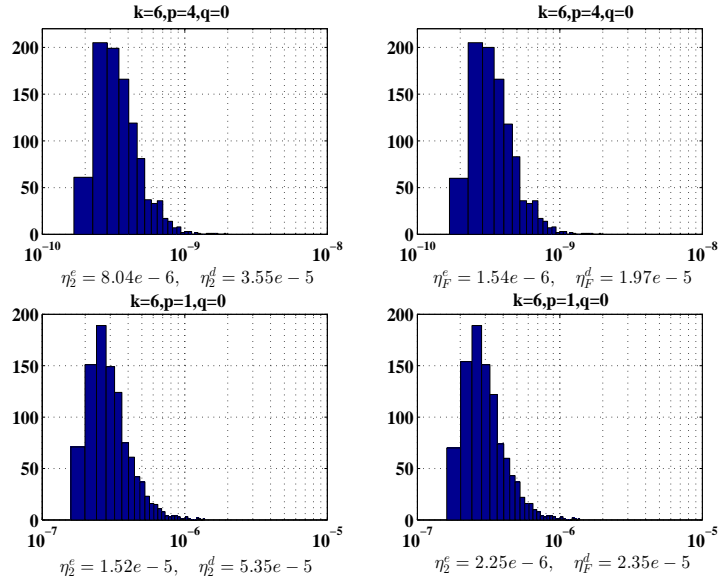}
\caption{\it  Approximation errors and upper bounds  for a  $100\times 80$ matrix whose singular values decay very fast (decay rate: 0.1) The left figures are for the estimates of spectral errors, while the right ones correspond to   Frobenius errors.  }\label{Fig5.2}
 \end{figure}

In Figure \ref{Fig5.1}, it is observed that for case (1) with slow decay rate in the singular values,  the upper bounds $\eta_2^e$ and $\eta_2^d$ are respectively about 15 and  40 times  the actual values of $\|\widehat {\bf A}_{k+p}^{(0)}-{\bf A}\|_2$,
while for the Frobenius error $\|\widehat {\bf A}_{k+p}^{(0)}-{\bf A}\|_F$, the estimated upper bounds $\eta_F^e$ and $\eta_F^d$ are much tighter, and they are only about 2 and 10 times the actual values, respectively.

In Figure \ref{Fig5.2} and  for case (2) with fast decay rate in the singular values, a relative large oversampling size $p=4$ gives   upper bounds that are not sharp enough, and there may be a factor  ${\cal O}(10^4)$
between the estimated upper bounds and actual approximation errors. When we take $p=1$, the estimates for the upper bounds have been greatly enhanced.
The reason is that the tested matrix ${\bf A}$ has fast decay rate in its singular values, therefore the orthonormal basis of ${\cal R}({\bf A\Omega})$ gives a good approximation of an $\ell$-dimensional ($\ell=k+p$)  left dominant singular subspace of ${\bf A}$,
 which makes
 $\|\widehat {\bf A}_{k+p}^{(0)}-{\bf A}\|_2\approx \sigma_{k+p+1}$, and when $p=4$, it is much smaller than the estimated bound $\eta_2^e\approx {\cal O}(\sigma_{k+1})$.

Overall, the test results in Figures \ref{Fig5.1}-\ref{Fig5.2} illustrate the rationality of theoretical estimates for approximation errors.}

\end{example}

\begin{example}\label{e:5-2} {\rm In this example, we test how different values of   $q$ in the power scheme affect  the approximation errors $\|\widehat {\bf A}_k^{(q)}-{\bf A}\|_a$.
We use standard test image {\sf lena512}\footnote{{\sf lena512}: https://www.ece.rice.edu/$\sim$wakin/images/} with $512\times 512$ pixels. This color image is characterized by a $512\times 512$ pure quaternion matrix ${\bf A}$ with entries
 ${\bf A}_{ij}=R_{ij}{\bf i}+G_{ij}{\bf j}+B_{ij}{\bf k}$,
where $R_{ij}, G_{ij},  B_{ij}$ represent  the red, green and blue pixel values at the location $(i,j)$ in the image, respectively.
The singular values and adjacent singular value ratio ${\sigma_{k+1}/\sigma_k}$ of ${\bf A}$ are depicted in Figure \ref{Fig5.3}.

Based on the structure-preserving  quaternion  Householder QR  and QMGS processes for getting  the orthonormal basis matrix ${\bf Q}$,
 we take the oversampling   $p=4$ and depict the approximation errors $\|\widehat {\bf A}_k^{(q)}-{\bf A}\|_a$ for  $k$ ranging from 5 to 200 with step 5 in Figures \ref{Fig5.4}--\ref{Fig5.5}, where {\sf svdQ}
plots  the  optimal rank-$k$ approximation  errors obtained via  the structure-preserving QSVD algorithm \cite{wl}.

\begin{figure}
\centering
\includegraphics[width=0.8\textwidth,height=5cm]{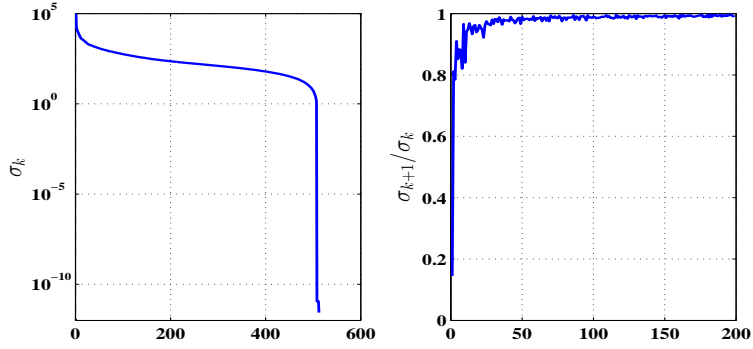}
\caption{\it  Singular values and adjacent singular value ratios for color image {\sf lena512}.}\label{Fig5.3}
 \end{figure}

\begin{figure}
\centering
\includegraphics[width=7cm,height=5cm]{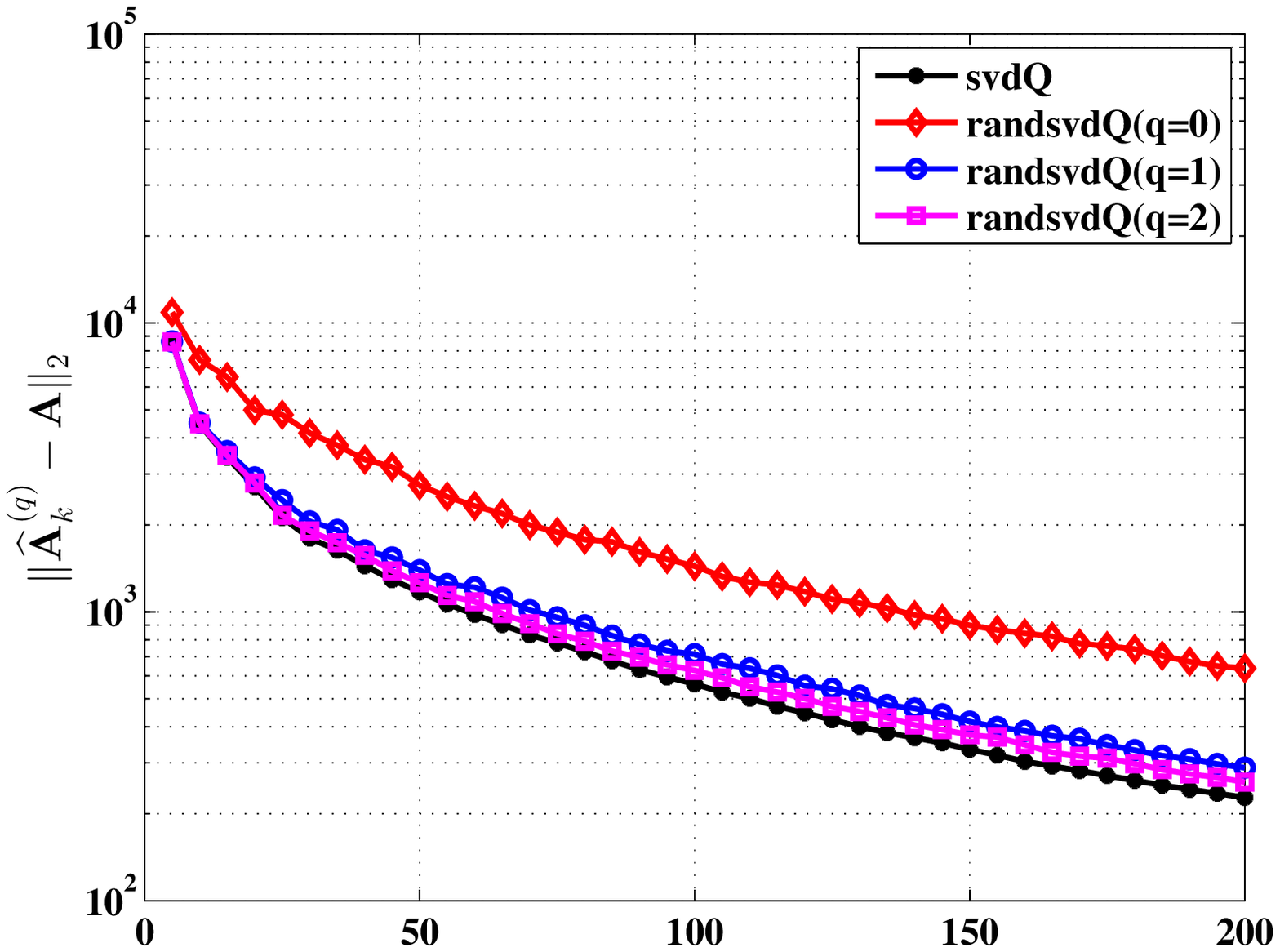}
\includegraphics[width=7cm,height=5cm]{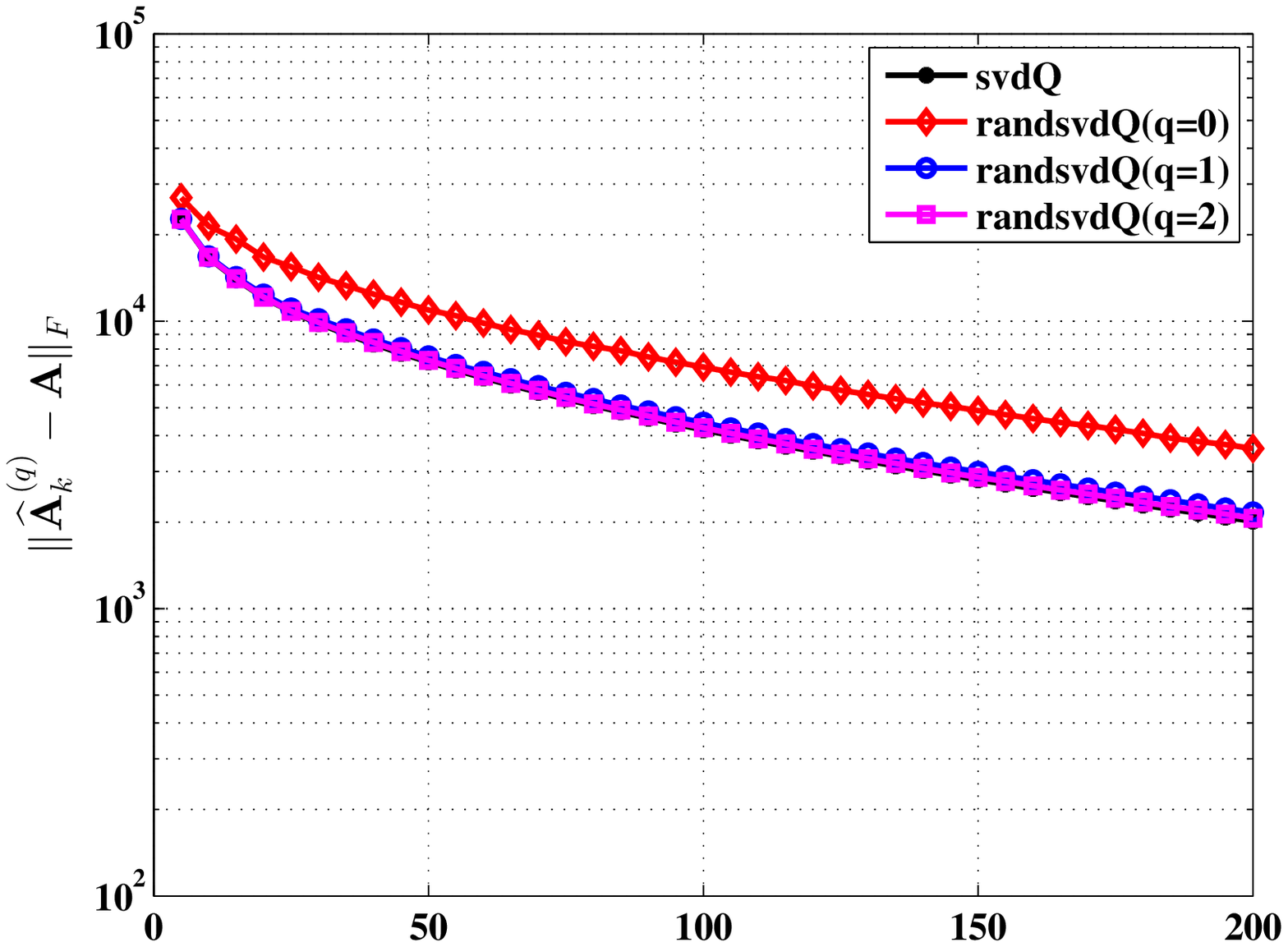}
\caption{\it Errors incurred for different power schemes, in which the orthonormal basis ${\bf Q}$ in {\sf randsvdQ} is obtained via   quaternion Householder QR procedure.}\label{Fig5.4}
 \end{figure}

 It is observed that when $k\ge 5$, the adjacent singular value ratio is greater than 0.8, the power scheme with $q=0$ gives the worst estimates for the rank-$k$ approximation errors among three cases. In the quaternion Householder QR-based algorithm,  the case with $q=2$ behaves better than that for $q=1$, since a smaller
  adjacent singular value  ratio  $\left({\sigma_{k+1}\over \sigma_k}\right)^{2q+1}$ of $({\bf AA}^*)^q{\bf A}$ helps   generate better basis matrix ${\bf Q}$ and  rank-$k$ matrix approximation.
 Although the approximation errors from randomized algorithms are not as accurate as the {\sf svdQ}-based ones, they still deliver
 acceptable peak signal-to-noise ratio ({\sf PNSR}) and relative approximate errors as listed in Table \ref{tab5.1}, in which the PSNR is defined by
 $$
 {\sf PSNR}(\widehat {\bf A}_k^{(q)}, {\bf A})=10\log_{10}{255^2mn\over \|\widehat {\bf A}_k^{(q)}-{\bf A}\|_F^2}.
 $$
It is observed that $q=1$ is acceptable for the desired accuracy.

\renewcommand\tabcolsep{20.0pt}
\begin{table}
\begin{center}
\caption{The peak signal-to-noise ratio and relative approximating errors for {\sf randsvdQ}}\label{tab5.1}
\begin{tabular}{*{28}{c}}
\hline
$k$& $q$&{\sf PNSR}&  ${\|\widehat {\bf A}_k^{(q)}-{\bf A}\|_2\over \|{\bf A}\|_2}$&${\|\widehat {\bf A}_k^{(q)}-{\bf A}\|_F\over \|{\bf A}\|_F}$\\\hline
50&1 &   24.7780 &   0.0115  &  0.0602\\
&2&  25.0501 & 0.0106& 0.0583\\\hline
100&1&29.4102& 0.0057& 0.0353\\
  & 2&29.7303& 0.0051 & 0.0340\\\hline
150&1&32.8041&0.0035&0.0239\\
&2&33.1368&0.0032& 0.0230\\\hline
\end{tabular}
\end{center}
\end{table}


\begin{figure}
\centering
\includegraphics[width=7cm,height=5cm]{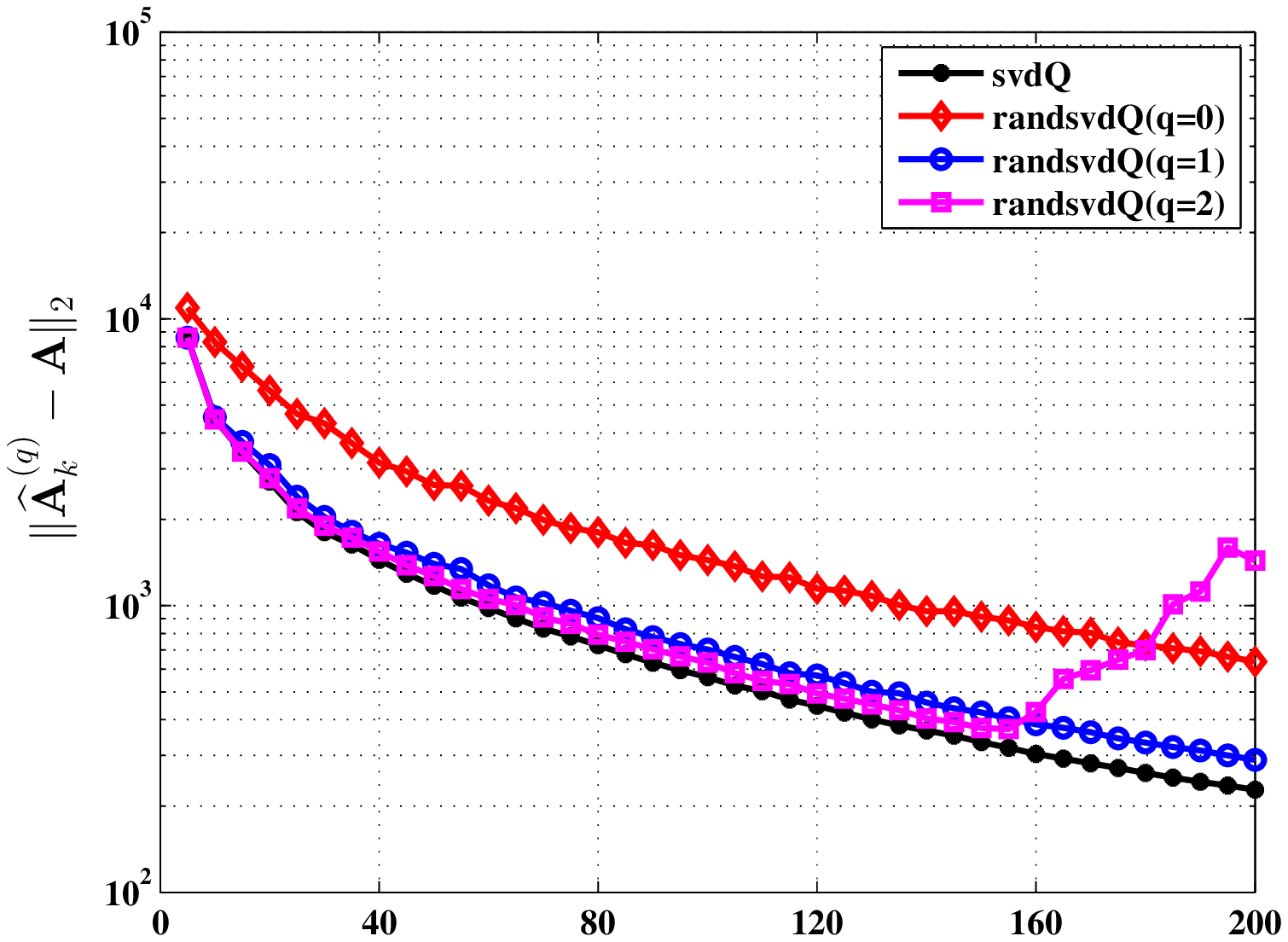}
\includegraphics[width=7cm,height=5cm]{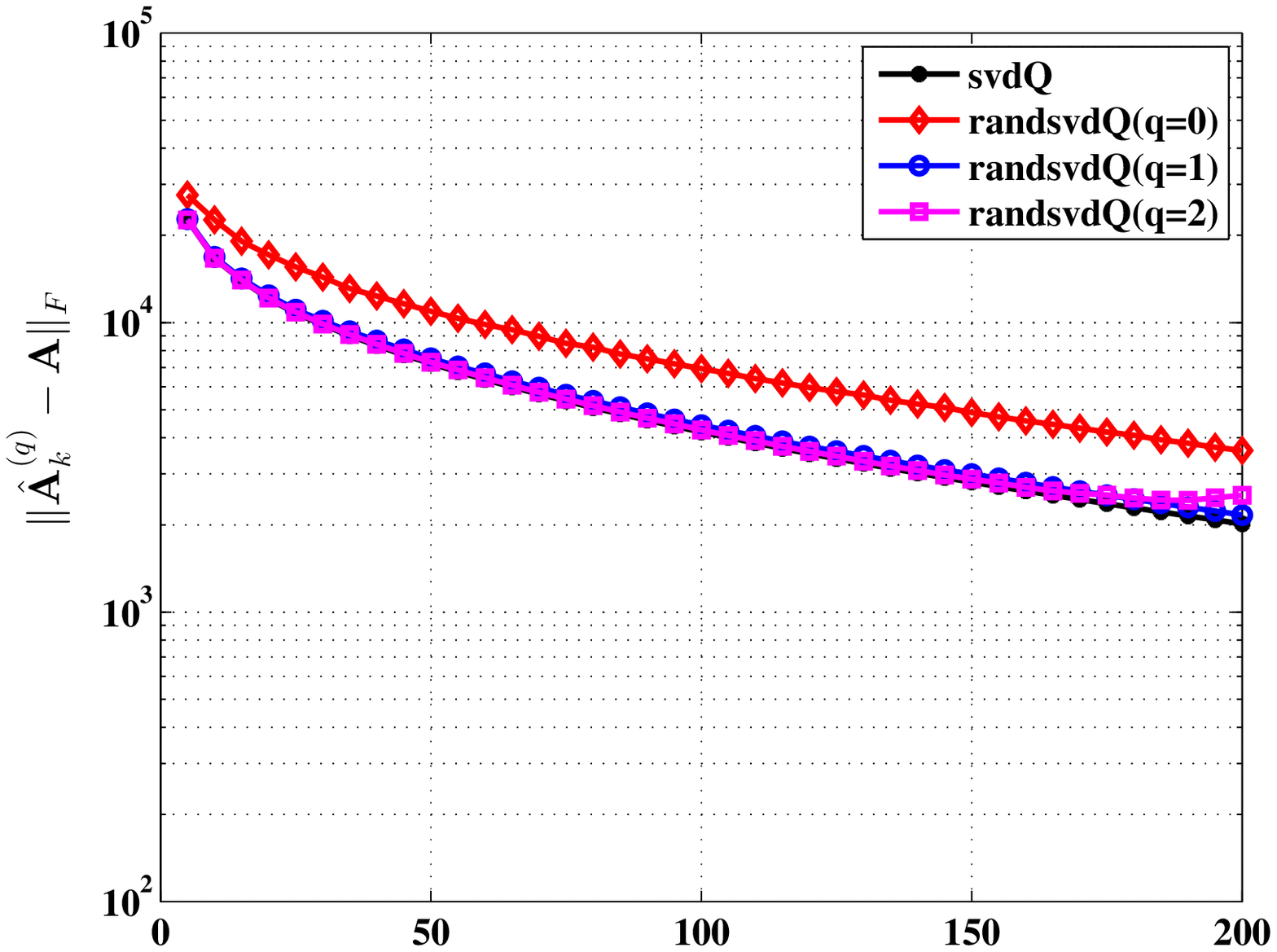}
\caption{\it Errors incurred for different power schemes, in which the orthonormal basis ${\bf Q}$ in {\sf randsvdQ} is obtained via quaternion MGS.}\label{Fig5.5}
 \end{figure}

 In Figure \ref{Fig5.5},   QMGS-based method is compared with quaternion Householder QR procedure. QMGS  gives satisfactory approximations for $k<160$ and $q=1$ or 2,
 while for $q=2$ and $k\ge 160$, the estimates become worse. That is partly because  for $q=2$,
 $\left({\sigma_{1}\over \sigma_{165}}\right)^{2q+1}=1.1e+13$ and  ${\bf Y}_q=({\bf AA}^*)^q{\bf A\Omega}$ tends to be an ill-conditioned matrix, which leads to  a great loss of orthogonality in the matrix ${\bf Q}$ during the QMGS procedure. However, the low-rank approximation problem only captures the dominant SVD triplets,    the target rank is usually small, and in the randomized algorithm we usually deal with the QMGS of a well-conditioned matrix,   the  QMGS with  $q=1$    is preferred,  since  it is more efficient than the  quaternion Householder QR.
 }

\end{example}

\begin{example}\label{e:5-3} {\rm In this example, we compare    numerical behaviors of    {\sf randeigQ} and {\sf prandsvdQ} algorithms in computing the rank-$k$ approximation of a large quaternion Hermitian matrix.
It is well known that the real Laplacian matrix plays important roles in image denoising, inpainting  problems for  the grayscale image.  Recently in \cite{ba}, complex Laplacian matrix is also discussed  in the mixed graph with some   directed and some undirected edges, and its zero eigenvalue   is proved to be related to the connection of the mixed graph.
Our example involves a quaternion graph Laplacian matrix for a color image, which is modified from  real \cite{hmt} and complex cases.

For this purpose, we  begin resizing {\sf lena512}  to a $60\times 60$-pixel color image, owing to the restricted memory of Laptop.
For each pixel $i$ in color channel  $s\in\{r, g, b\} $,  form a vector
$x_s^{(i)}\in {\mathbb R}^{25}$ by gathering the 25 intensities of the pixels in a $5\times 5$ neighborhood
centered at pixel $i$. Next, we form a $3600\times 3600$ pure quaternion Hermitian weight matrix ${\bf W}= W_{r}{\bf i}+ W_{ g}{\bf j}+  W_{ b}{\bf k}$
with $ {\bf w}_{ji}=  {\bf w}_{ij}^*$, ${\bf w}_{ii}=0$, and
$ {\bf w}_{ij}=\left({ w}_r\right)_{ij}{\bf i}+\left({ w}_g\right)_{ij}{\bf j}+\left({w}_b\right)_{ij}{\bf k}$ for $i<j,$ which is determined by
$$
\left(w_s\right)_{ij}={\rm exp}\left\{-\|x_s^{(i)}-x_s^{(j)}\|_2^2/\sigma_s^2\right\},\quad  j>i,\quad s\in \{r, g, b\}.
$$
Here  the entries in their strictly upper triangular part  of $W_s$
 reflect the similarities between patches, and the parameter $\sigma_s$ controls the level of sensitivity in each channel.
By zeroing out all entries  of skew-symmetric matrices $W_r,W_g$ and $W_b$ except the four largest ones in magnitude in
each row, we obtain  sparse   weight matrices $\widetilde W_s$ and $\widetilde {\bf W}$.
Similar to the complex case, let $D$ be a diagonal matrix with  $d_{ii}=\sum\limits_{j}|{\bf w}_{ij}|$, and define  the quaternion Laplacian matrix ${\bf L}$ as
$$
{\bf L}=I-D^{-1/2}\widetilde{\bf W}{ D}^{-1/2}.
$$
For all $s\in \{r, g, b\}$, take  $\sigma_s=50$, store the $14400\times 3600$ real matrix  ${\bf L}_{\rm c}$, and use structure-preserving algorithm {\sf eigQ} \cite{jwl} to compute all eigenvalues of ${\bf L}$. Here the  Hermitian matrix {\bf L} is a very extreme case with positive eigenvalues, and the smallest ratio ${\sigma_{k+1}/\sigma_k}$ of adjacent eigenvalues (singular values) of ${\bf L}$ is greater than 0.98.

Take $k=200, p=10, q=0,1,2$ to compare the eigenvalues of ${\bf L}$ via {\sf randeigQ}, {\sf prandsvdQ}.
In all cases, the approximations of eigenvalues are not good enough, because  $k=200$ only captures  less than 10\% proportion of eigenvalues in this extreme case, as revealed in the
left figure of   Figure \ref{Fig5.6}.
Due to the quite slow decay rate of eigenvalues,  when $q$ is small, say for $q=0$,  the eigenvalues computed via {\sf randeigQ}, {\sf prandsvdQ} are not accurate enough, but
{\sf prandsvdQ} still approximates eigenvalues  better than {\sf randeigQ}, as predicted in Remark \ref{r:3-3}.
  The accuracy is improved as $q$ increases, and for this extreme example,  $q=2$ is sufficient to guarantee the eigenvalues  from two algorithms  with almost the same accuracy.
 For general cases,  we believe that {\sf randeigQ} is as reliable as {\sf prandsvdQ} but more efficient for practical low-rank Hermitian matrix approximation problems with  dominant singular values.

\begin{figure}
\centering
\includegraphics[width=7.2cm,height=5cm]{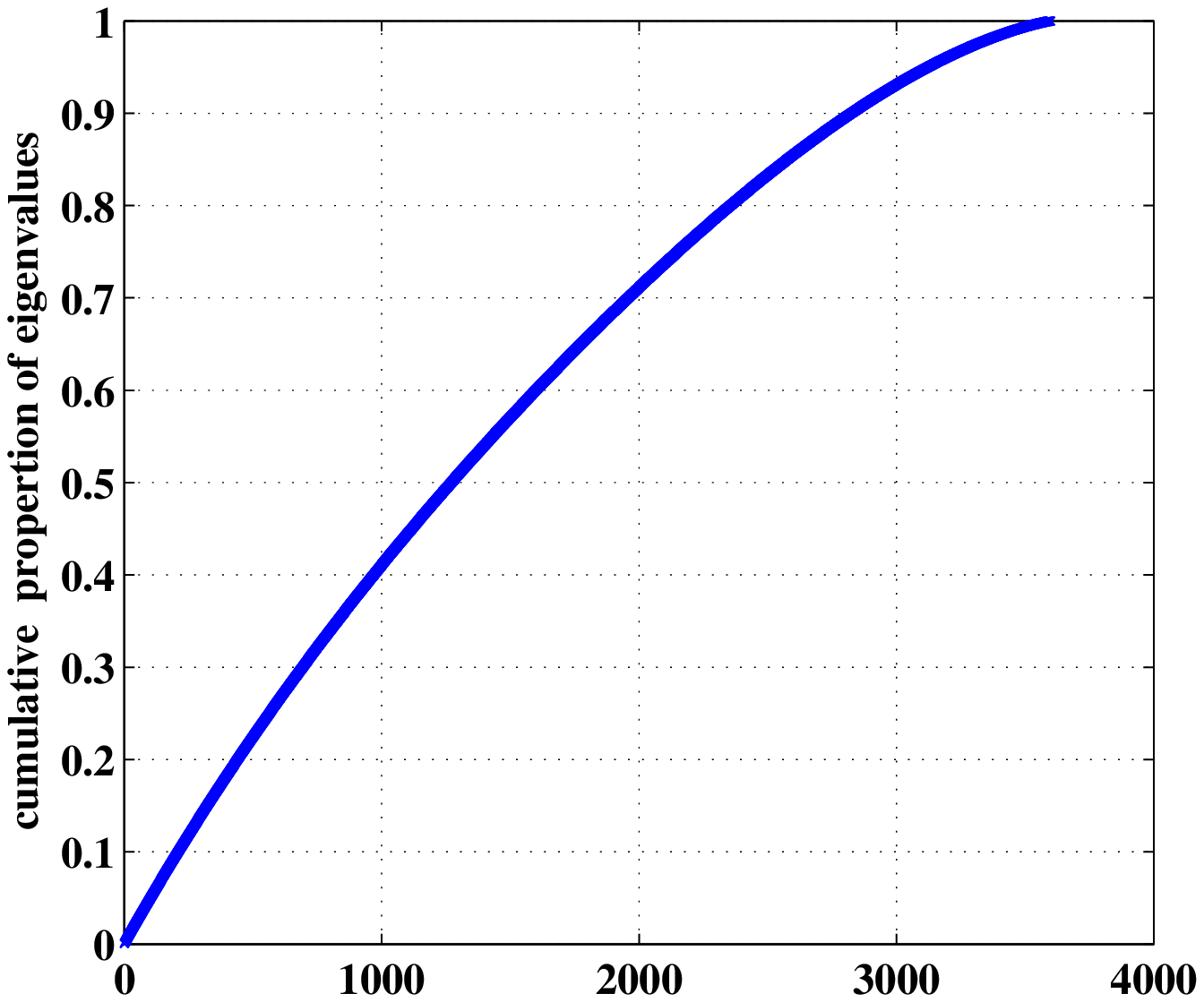}
\includegraphics[width=7cm,height=5cm]{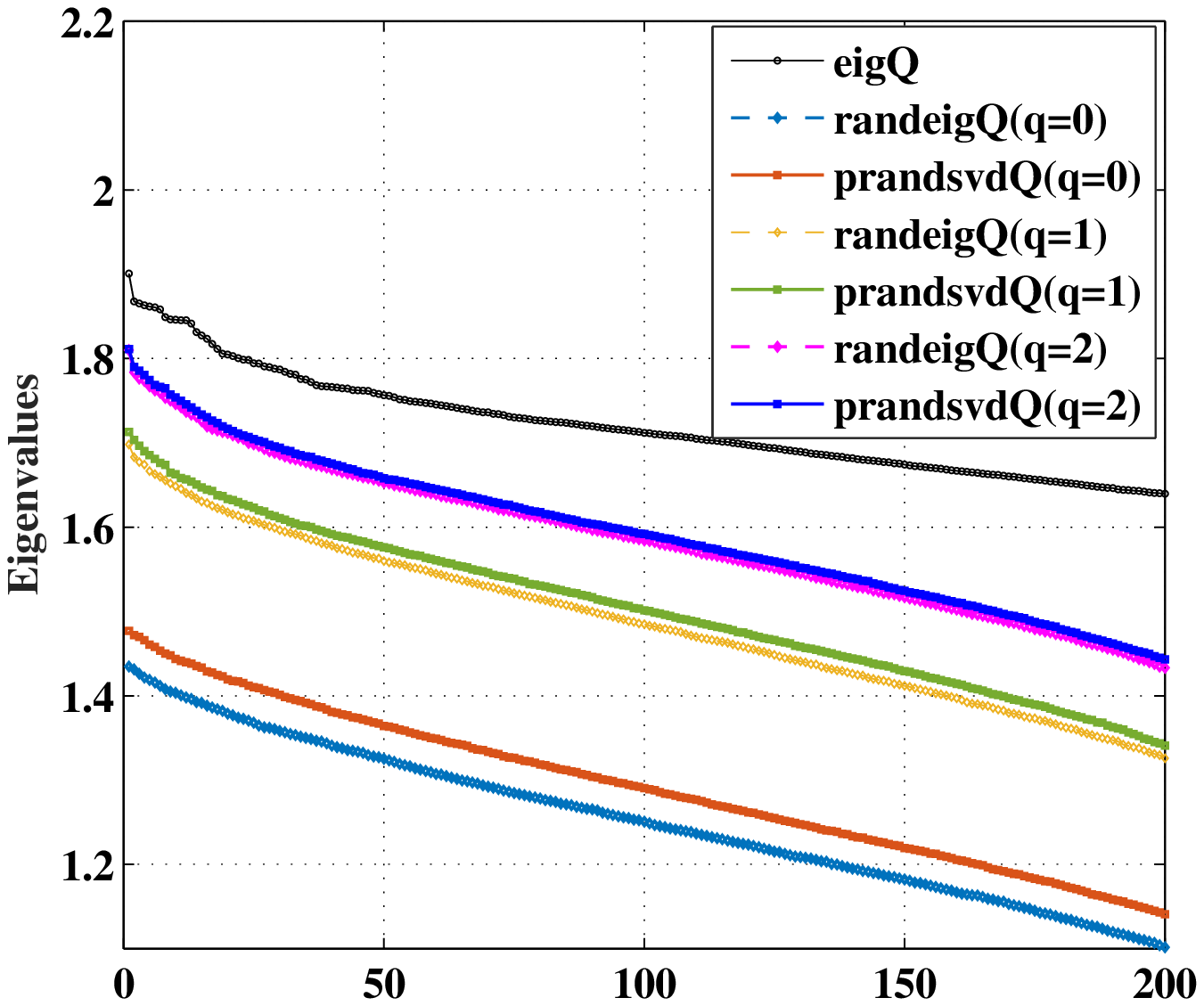}
\caption{\it The cumulative proportion of eigenvalues of  a quaternion Laplacian matrix   and eigenvalues computed via {\sf randeigQ} and {\sf prandsvdQ} for $k=200, p=10$.}\label{Fig5.6}
 \end{figure}

}

\end{example}

\begin{example}\label{e:5-4} {\rm In this example, we consider  the color  face recognition problem \cite{jns} based on color principal component analysis (CPCA) approach.   Suppose that there are $s$ training color image samples, denoted by $m\times n$ pure quaternion matrices ${\bf F}_1, {\bf F}_2, \ldots, {\bf F}_s,$ and the average is ${\bf \Psi}={1\over s}\sum\limits_{t=1}^s {\bf F}_t\in {\mathbb Q}^{m\times n}$. Let
$
{\bf X}=[{\rm vec}({\bf F}_1)-{\rm vec}({\bf \Psi}),\cdots, {\rm vec}({\bf F}_s)-{\rm vec}({\bf \Psi})],
$
where ${\rm vec}(\cdot)$ means to stack the columns of a matrix into a single long vector. The core work of CPCA approach is to compute the left singular vectors corresponding to the first $k$ largest singular values of ${\bf X}$, which are called the eigenfaces. The eigenfaces can also be obtained from
the {\sf eigQ}  algorithm  \cite{jwl} applied to ${\bf XX}^*$ or ${\bf XX}^*$.

\begin{figure}
\centering
\includegraphics[width=7cm,height=5cm]{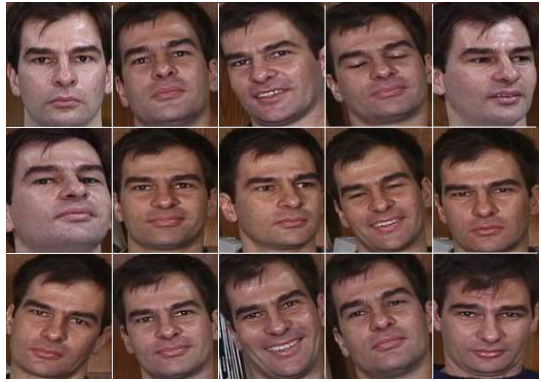}
\caption{\it Sample images for one individual of the Georgia Tech face database }\label{fig21}
 \end{figure}

\begin{figure}
\centering
\includegraphics[width=7cm,height=5cm]{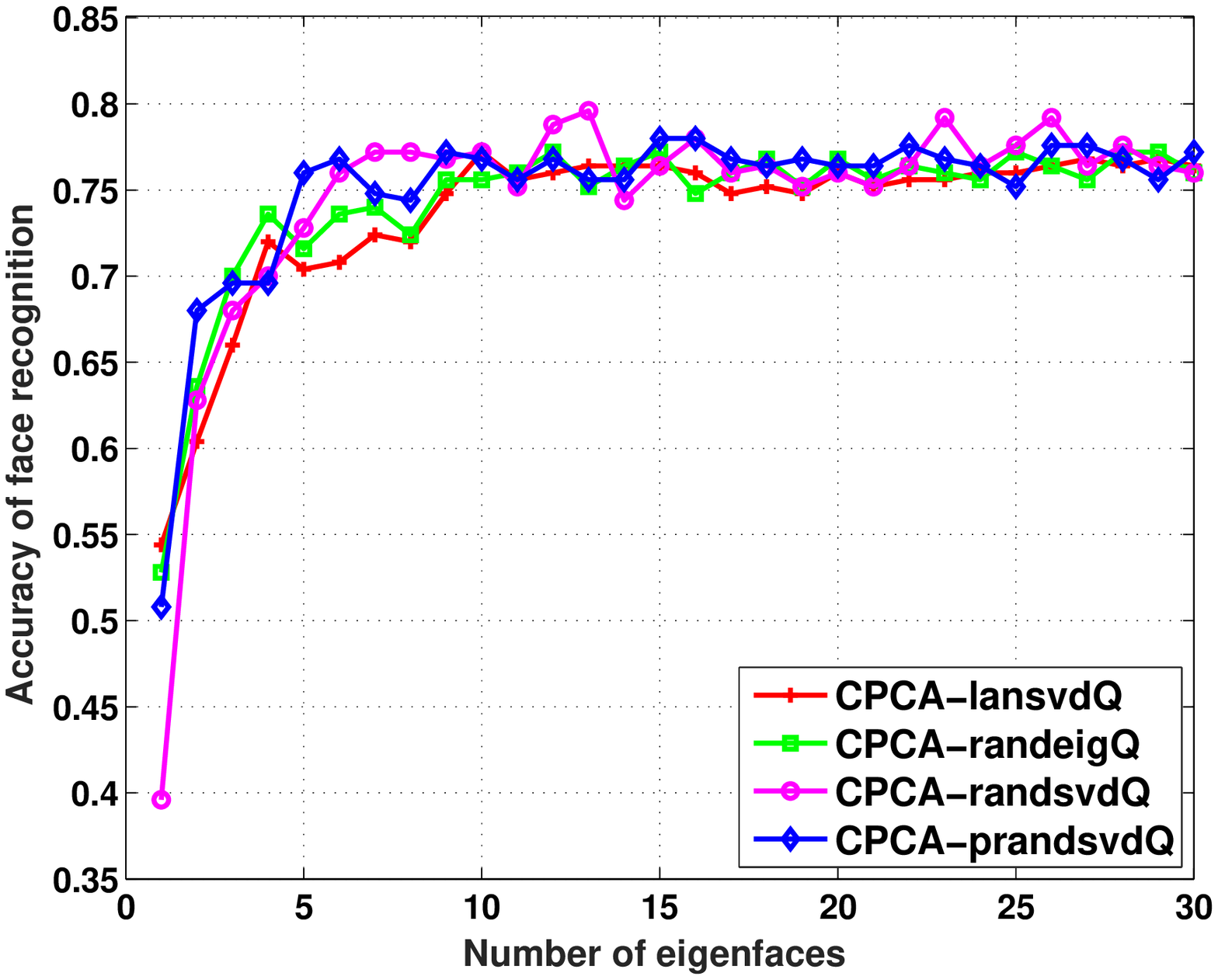}
\includegraphics[width=7cm,height=5cm]{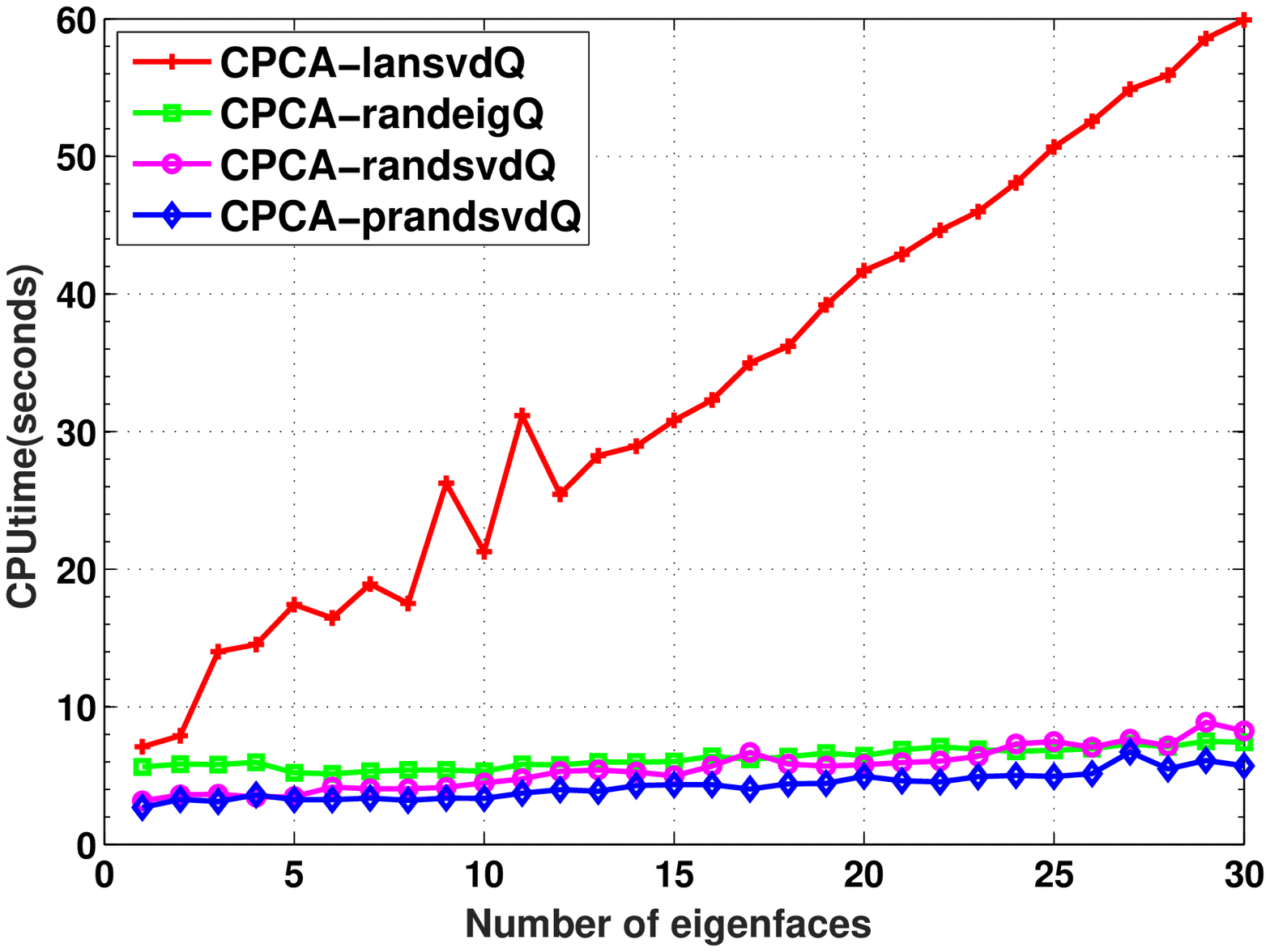}
\caption{\it The  color face recognition accuracy  and   CPU time     by {\sf lansvdQ}, {\sf randsvdQ}, {\sf randeigQ} and  {\sf prandsvdQ} methods with parameters $p=4,q=0$.   }\label{fig3}
 \end{figure}

For color image samples, we use the Georgia Tech face database\footnote{The Georgia Tech face database. http://www.anefian.com/research/face\_reco.htm}, and
all images  are manually cropped, and then resized to $120\times 120$ pixels. The samples of the cropped images are shown in Figure \ref{fig21}.  There
are 50 persons to be used. The first ten face images per individual person are chosen
for training and the remaining five face images are used for testing. The number of
chosen eigenfaces, $k$, increases from 1 to 30. We need to compute $k$
SVD triplets of a $14400\times 500$ quaternion matrix ${\bf X}$, in which the 14400 rows refer to
$120 \times 120$ pixels and the 500 columns refer to 50 persons with 10 faces each.

As  revealed in \cite{jns},  the matrix is very large and the {\sf svdQ}  algorithm does not finish the computation of the singular
value decomposition of ${\bf X}$ in 2 hours and {\sf eigQ} needs about seven times of the running CPU time via the quaternion Lanczos-based algorithm ({\sf lansvdQ})\footnote{https://hkumath.hku.hk/$\sim$mng/mng\_files/LANQSVDToolbox.zip}.
In this experiment   we consider the {\sf lansvdQ},  {\sf randsvdQ},  {\sf prandsvdQ}  algorithms of  ${\bf X}$, and {\sf randeigQ} algorithm of ${\bf X^*X}$, where the orthonormal basis  is derived based on   quaternion MGS process, and in {\sf randeigQ}, the matrix ${\bf X^*X}$ is not explicitly formed.
  The detailed comparisons of recognition accuracy and running CPU time  of candidate methods are depicted in Figure \ref{fig3}, in which the accuracy of face recognition is
the percentage  of  correctly recognized persons for given 250 test images.
For $p=4$ and $q=0$,   randomized algorithms have higher recognition accuracy than  {\sf lansvdQ}, and   are much more efficient than {\sf lansvdQ}.
 Moreover,   the preconditioning technique for {\sf randsvdQ} can slightly enhance the efficiency of the algorithm.
Unlike {\sf lansvdQ}, the CPU time for randomized algorithms does not increase significantly with the target rank (number of eigenfaces).
 {\sf lansvdQ} is much less efficient partly because it uses for-end loop and performs  matrix-vector products at each iteration, while the randomized
algorithms  make full use of the  matrix-matrix products that have been highly optimized for maximum efficiency on modern serial and parallel architectures \cite{gv2}.}

\end{example}

\begin{example}\label{e:5-5} {\rm In this example, we generalize the  fast frequent directions  via subspace embedding ({\sf SpFD}) method  \cite{tc} to the quaternion case. The corresponding algorithm is referred  to as {\sf SpFDQ}, and is compared with  {\sf prandsvdQ}   through the color face recognition problem in  {Example} \ref{e:5-4}.

\begin{figure}[h]
\centering
\includegraphics[width=7cm,height=5cm]{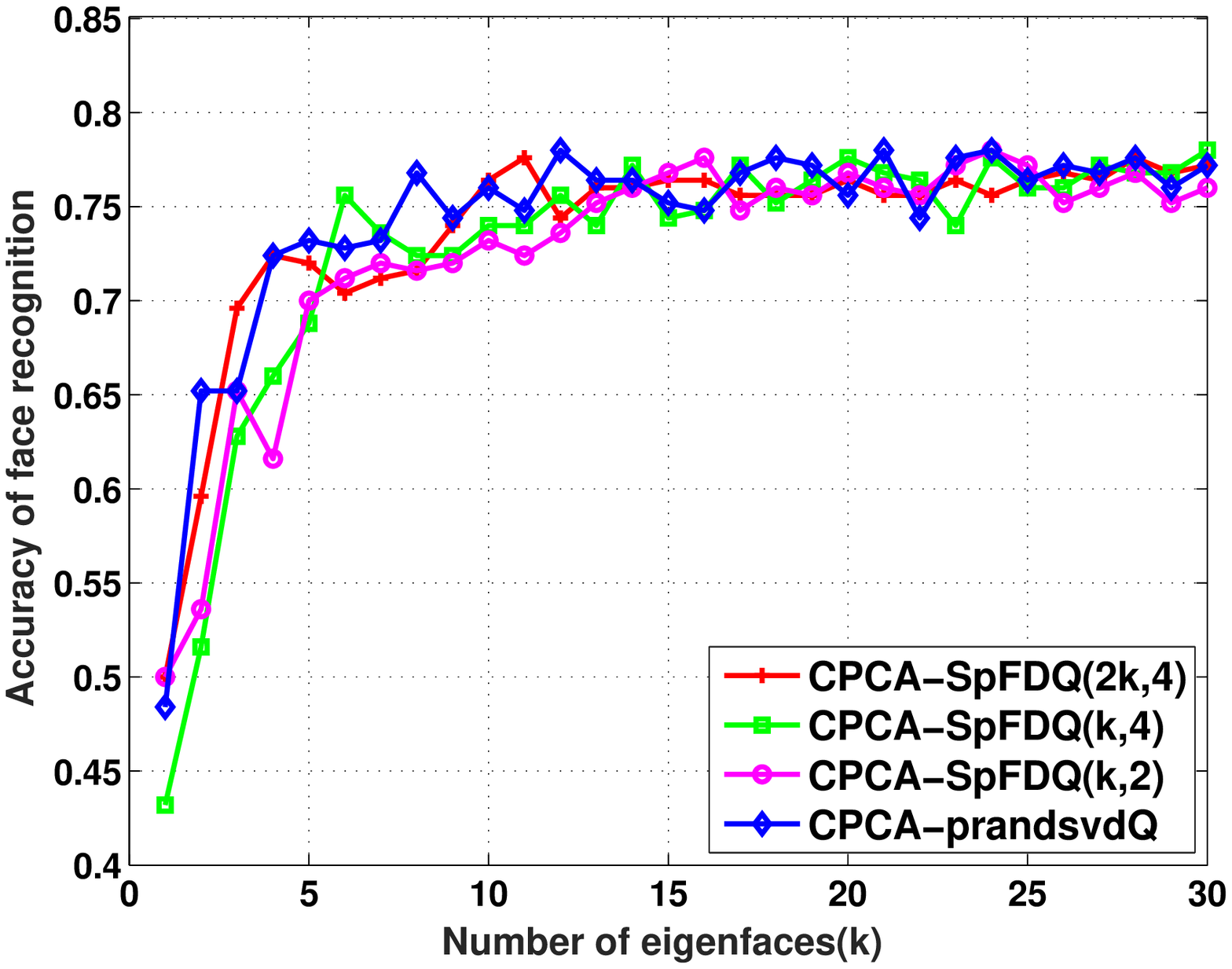}
\includegraphics[width=7cm,height=5cm]{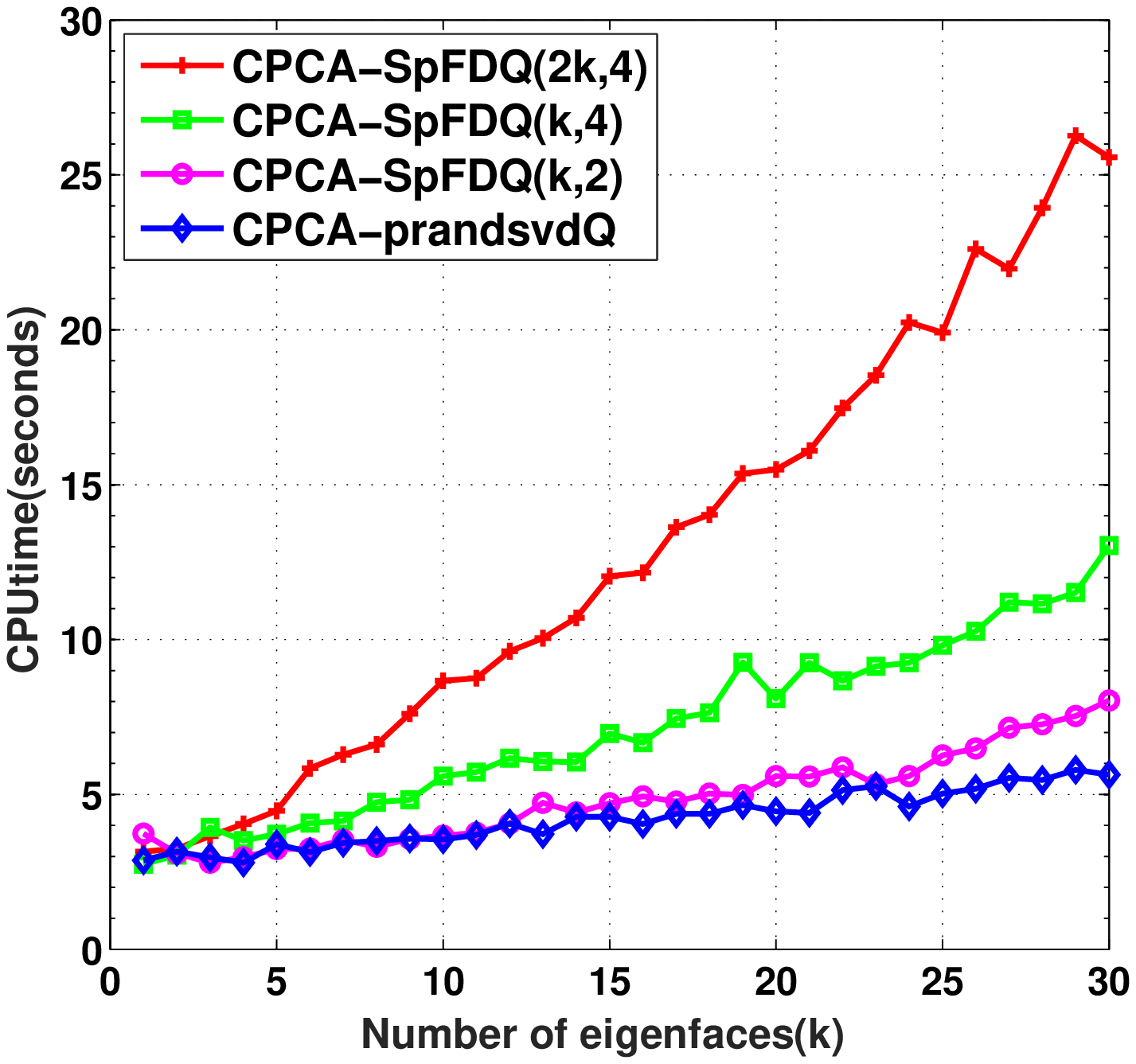}
\caption{\it The  color face recognition accuracy  and   CPU time     by {\sf SpFDQ}($\ell,t$) and  {\sf prandsvdQ} methods with parameters $p=4,q=0$.   }\label{fig4}
 \end{figure}

 Given  a real matrix $A\in {\mathbb R}^{m\times n}$ ($m\ge n$),  the {\sf SpFD($\ell,t$)} algorithm  squeezes  the rows of $A$ by pre-multiplying  $SP$ on $A$, where $t$ is  assumed to be a factor of $m$ (if not, append zero rows to the end of $A$ until $m$ is),
$P$ is a random permutation matrix, and $S={\rm diag}(S_1,\ldots, S_t)$  is a sparse sketching matrix with $S_i\in {\mathbb R}^{\ell\times {m\over t}}$  being generated on a probability distribution. At the start of the algorithm, it extracts and shrinks  the top $\ell$ important right singular vectors of   a two-layered matrix $\Big[{S_{1}PA\atop S_2PA}\Big]$ via SVD, and then  combines them with the next $\ell$ rows in $SPA$ to form a new two-layered matrix. Repeat the procedure until the last $\ell$ rows of $SPA$ is combined into the computation. Finally, an orthonormal basis $V_\ell\in {\mathbb R}^{n\times \ell}$ for the  row space of $SPA$ is obtained, and a rank-$k$ approximation of $A$ is derived based on the SVD of $AV_\ell$. The algorithm consists of $(t-1)$ iterations, and the total cost is
$$
2{\rm nnz}(A)(\ell+1)+[24n\ell^2+160\ell^3](t-1)+6m\ell^2+20\ell^3+2m\ell k+2mn\ell,\quad t>1,
$$
where $m\ge n\ge \ell\ge k$, $m\gg \ell$. The choice of $t=2, \ell=k$ corresponds to an algorithm with the cheapest cost, while for $t=\lceil m/\ell\rceil$, {\sf SpFD}($\ell,t$)  reduces to a slight modification of {\sf FD} in \cite{glp}.

In the  {\sf SpFDQ}($\ell, t$) algorithm,   ${\bf A}$ is taken to be the $14400\times 500$ matrix ${\bf X}$ in  Example \ref{e:5-4}, and
the choice of sketching matrix $S$ is the same as the real case. To perform a fair comparison, we also consider the preconditioned technique in the QSVD of a short-and-wide or tall-and-narrow quaternion matrix.  During the $(t-1)$ rounds of QSVD in the iteration,
due to the potential singularity of the sketching matrix $S_i$ that might lead to  a singular   two-layered matrix,
we  apply quaternion Householder QR first and then implement  the QSVD on a small-size matrix.
In the last round of QSVD of ${\bf A V}_\ell$,   the QSVD of ${\bf AV}_\ell$ is obtained via the QMGS of ${\bf AV}_\ell$ first and then applying QSVD to a small upper triangular factor.

The accuracy of face recognition and running CPU time of     {\sf SpFDQ}($\ell,t$) and {\sf prandsvdQ} algorithms   are shown in Figure \ref{fig4}. The depicted results demonstrate that {\sf SpFDQ}($k,2$) is the most efficient one among all {\sf SpFDQ}($\ell,t$) algorithms, while
{\sf prandsvdQ}  is a little more efficient than {\sf SpFDQ}($k,2$) when $k$ increases.
For the recognition accuracy,  {\sf prandsvdQ} has higher recognition accuracy for most  parameter values of $k$, while there also exists a parameter, say for $k=22$,
{\sf prandsvdQ} has lower recognition accuracy  than other candidate methods. That is partly because the sketching matrix $S$ and random   ${\bf \Omega}$ are randomly generated on specific distributions, and the recognition accuracy  is sometimes affected by the properties of some specific  random matrices.

In order to perform a fair comparison, in Table \ref{Table2} we execute each algorithm 20 times, and display the average (avrg), maximal (max) and minimal (min) numbers of correctly recognized persons among 250 test faces for 50 persons, and the average running CPU time (avtime) is also given. It is observed that when $k$ is small, say for $k\le 9$, there exist big fluctuations on the recognition accuracy of {\sf SpFDQ}($k$,2), and   the average numbers of recognized faces
increase when the sketching size in {\sf SpFDQ}($2k$,2) is increased, but  {\sf SpFDQ}($2k$,2) still has lower recognition accuracy than  {\sf prandsvdQ}. When $k$ increases,
the difference of face recognition accuracy becomes smaller, while for the running time, {\sf prandsvdQ} is the most efficient.

\renewcommand\tabcolsep{5.0pt}
\begin{table}[!htb]
\begin{center}
\caption{Comparisons of {\sf SpFDQ}($\ell$,2) with {\sf prandsvdQ} for PCA-based color image recognition problems}
\label{Table2}
\begin{tabular}{*{28}{c}}
\hline
 {\sf SpFDQ}($k$,2)\\
 \hline
 $k$& 3&6&9&12&15&18&21&24&27&30\\
avrg& 153.25&178.65&  184.70&  188.05&  189.80&  189.85&  190.00&  190.80&  190.95&  192.40\\
 max &184&188&191&  194& 195&  193&  195& 195& 195&  196\\
  min&130& 169& 176&  182& 183 & 185&184& 187& 187&  187\\
avtime  &2.97&    3.29& 3.68& 4.10&4.63&    5.47&5.62& 6.02& 6.86&7.45\\
\hline
 {\sf SpFDQ}($2k$,2)\\
 \hline
 $k$& 3&6&9&12&15&18&21&24&27&30\\
avrg& 161.35&178.75&  184.85&  190.30&  189.25&  190.00&  188.95&  190.60&  190.85&  191.80\\
 max &172&186&190&  194& 192&  192&  192& 193& 196&  196\\
  min&150& 172& 180&  186& 186 & 187&186& 188& 186&  188\\
avtime  &3.18&    3.92& 4.72& 5.60&6.96&    7.91&9.27& 10.47& 12.09&13.56\\
\hline
{\sf prandsvdQ}\\
\hline
$k$& 3&6&9&12&15&18&21&24&27&30\\
avrg&174.35&  187.10&  190.55&  191.30&  190.65&191.15&  192.65&  193.45&192.25&192.85\\
 max &182&  195&201&198&194&198&196&198&197&197\\
 min &164&182&183&185&185&184&187&188&188&189\\
 avtime  &3.02&3.30&3.44&3.70&4.11&4.44&4.73&5.02&5.50&5.82\\\hline
  \end{tabular}
\end{center}
\end{table}

}

\end{example}

\section{Conclusion}

In this paper we have presented  the randomized QSVD   algorithm for   quaternion low-rank matrix approximation problems.  For large scale problems with a small target rank,  the randomized algorithm  compresses  the size of  the input matrix by the quaternion normal distribution-based random sampling,  and approximates dominant SVD triplets with good accuracy and high efficiency.
 The approximation errors of the randomized algorithm are illustrated by the detailed theoretical analysis and numerical examples.
 Compared to the  Lanczos-based QSVD ({\sf lansvdQ})  and fast frequent direction via  subspace embedding  ({\sf SpFDQ}) algorithms,
   the randomized  algorithms display  their effectiveness and efficiency  for PCA-based color image recognition problems.\\

{\bf Acknowledgments.} The authors are grateful to the handling editor and three
anonymous referees for their useful comments and suggestions, which greatly improved
the original presentation.


%

\end{document}